\documentclass{amsart}
\usepackage{amsfonts}

\usepackage{dsfont}

\newtheorem{lem}{Lemma}

\newtheorem{thm}{Theorem}

\DeclareMathOperator{\Tr}{Tr}
\allowdisplaybreaks
\begin{document}
 \title[Large gaps]{Large gaps of  CUE and GUE}

\author[Feng and Wei]{Renjie Feng and Dongyi Wei}
 \address{Sydney Mathematical Research Institute, The University of Sydney.}
\email{renjie.feng@sydney.edu.au}
\address{School of Mathematical Sciences, Peking University, Beijing, China, 100871.}

 \email{jnwdyi@pku.edu.cn}

\date{\today}
   \maketitle
	\begin{abstract}		In this article, we study the largest gaps of the classical random matrices of CUE and GUE, 
and show that after rescaling, the  limiting densities  are given by  the Gumbel distributions.
	\end{abstract} 
\section{Introduction}\label{intro}
	In random matrix theory, the typical spacings between eigenvalues of classical random matrices have been well understood for a long time \cite{AGZ, DG}. But there are only few results known for  the extremal spacings.
	The rescaling limits of the smallest gaps of CUE and GUE  (where the point processes of  eigenvalues are both determinintal point processes) were proved by Vinson and he also suggest the decay order of the largest gap \cite{V}.   Later on, Soshnikov studied the smallest gaps   for the general determinantal point processes with translation invariant kernels \cite{So}, and proved that the point processes of the smallest gaps after rescaling are asymptotic to the Poisson point processes. In \cite{BB}, Ben Arous-Bourgade derived the smallest gaps for CUE and GUE, and they further proved the decay order of the largest gap for these two ensembles which confirmed Vinson's prediction.  The proofs in \cite{BB, So, V} highly depend on the determinantal structures of the point processes.  For the point processes without determinantal structures, in  \cite{FW}, we developed a new technique based on the Selberg integral to prove the smallest gaps for the circular log-gas $\beta$-ensemble for any positive integer $\beta$. As special cases, our result implies the limiting distributions of the smallest gaps of the classical random matrices of COE, CUE and CSE. The same technique is further applied to GOE \cite{FTW}. Recently, a completely different approach has been employed to address the smallest gaps for GSE, leveraging its Pfaffian structure \cite{FLY}. 
	
	In this paper, we will derive the rescaling limits of the largest gaps of CUE and GUE by their determinantal structures. The main results are that the laws of the rescaling limits of the $k$-th largest gaps are given by  the Gumbel distributions for any fixed positive integer $k$. These results are further shown to be universal for general Wigner matrices \cite{SB, LLM}. 
	
	To state our results, let's first consider CUE. Let $u_n$ be a Haar-distributed  unitary
	matrix $U(n)$ over $\mathbb{C}^n$.  Suppose $u_n$ has eigenvalues $e^{i\theta_k}$'s with ordered eigenangles
	$0< \theta_1 < \cdots < \theta_n < 2\pi$. Let $m_1>m_2>\cdots$ be the largest gaps between successive eigenangles
	of $u_n$ i.e., $m_k\ (1\leq k\leq n)$ is the decreasing rearrangement of $\theta_{k+1}-\theta_k\ (1\leq k\leq n)$ with
	$ \theta_{k+n}=\theta_k+2\pi.$  In  \cite{BB},   Ben Arous-Bourgade proved that for any $p>0$ and $l_n=n^{o(1)},$  the following limit holds:
	$$\frac{nm_{l_n}}{\sqrt{32\ln n}}\buildrel\hbox{\footnotesize $L^p$}\over\longrightarrow 1.$$
	In this article, we  further derive the rescaling limit law for $m_k$ as follows.
	\begin{thm}\label{thm1}Let's denote $m_k$ as the $k$-th largest gap of CUE and $$ \tau_k=(2\ln n)^{\frac{1}{2}}(nm_k-(32\ln n)^{\frac{1}{2}})/4-(3/8)\ln(2\ln n),$$ then 
	the number of the rescaling point process $\{\tau_k\}_{k=1}^n$ falling in $[x,+\infty)$ tends to a Poisson random variable with mean $e^{c_1-x}$ for any fixed $x\in\mathbb R$. Here, $c_1=c_0+\ln(\pi/2) $ where $c_0=\frac1{12}\ln 2+3\zeta'(-1)$ and $\zeta(x)$ is the Riemann zeta function.  This further implies that,  for any bounded interval $ I\subset\mathbb{R}$ and any fixed positive integer $k$,  the limiting density for the $k$-th largest gap is given by the Gumbel distribution,
		$$\lim_{n\to+\infty}\mathbb P(\tau_k\in I)=\int_I\frac{e^{k(c_1-x)}}{(k-1)!} e^{-e^{c_1-x}}dx.$$ In particular, the limiting density for the largest gap $\tau_1$ is
		$$e^{c_1-x} e^{-e^{c_1-x}}.$$
		
	\end{thm}
	Let's sketch the main ideas to prove Theorem \ref{thm1}.  First, by the uniform asymptotic expansion \eqref{1} of the gap probability for a given arc of the circle to be free of eigenvalues,  we can find the correct rescaling formula for the largest gap $m_k$ and our crucial observation  is  the  rescaling limit  \eqref{2} in Lemma \ref{obs}.  For a sequence of decreasing point processes, we will first prove Lemma \ref{lem1} which provides a  criterion for the convergence of the number of points falling in $[x, +\infty)$   to a Poisson random variable.  Lemma \ref{lem1} implies that  Theorem \ref{thm1} will be proved by the upper bound \eqref{13} and the lower bound \eqref{14}. The upper bound can be proved by the negative association property of the determinantal point processes.
	The lower bound is  the most essential part of the whole proof, which is based on the  asymptotic splitting formula \eqref{42} for the gap probabilities in Lemma \ref{splt}. The proof of Lemma \ref{splt}  is further  based on Lemma \ref{lem15} and Lemma \ref{lem18} for the eigenvalue estimates of some symmetric operators.
	

	For GUE,  the joint density of the eigenvalues is \begin{align}\label{6}\frac{1}{Z_n}e^{-n\sum\limits_{i=1}^n\lambda_i^2/2}\prod_{1\leq i<j\leq n}|\lambda_i-\lambda_j|^2\end{align}with respect to the Lebesgue product measure  on the simplex $ \lambda_1<\cdots <\lambda_n.$ And the empirical spectral distribution   converges in probability to the
	semicircle law \cite{AGZ}\begin{align*}\rho_{sc}(x)=\sqrt{(4-x^2)_+}/(2\pi),\end{align*}where we denote $f_+:=\max(f,0).$
	
	For the largest gaps of GUE, the result is completely different inside the bulk and on the edge of the semicircle law. On the edge,  the largest gap is of order $n^{-2/3}$ which is indicated by the Tracy-Widom law \cite{AGZ}; while inside the bulk, the largest gap is of order $\sqrt{\log n}/n$ \cite{BB, V}. To be more precise,  given $I=[a,b] $ which is a compact subinterval of $(-2,2)$, let $m_1^*>m_2^*>\cdots$ be the largest gaps
	of type $ \lambda_{i+1}- \lambda_{i}$ with $\lambda_{i+1}, \lambda_{i}\in I$,  then Ben Arous-Bourgade \cite{BB} showed that for any $p>0$ and $l_n=n^{o(1)},$ the following limit holds: 
	$$\left(\inf_I\sqrt{4-x^2} \right)\frac{nm_{l_n}^*}{\sqrt{32\ln n}}\buildrel\hbox{\footnotesize $L^p$}\over\longrightarrow 1.$$
	Regarding the GUE case, we have \begin{thm}\label{thm2} Given  $I=[a,b] $ which is a compact subinterval of $(-2,2)$, let $m_k^*$ be the $k$-th largest gap of GUE falling in $I$,  we denote $S(I)=\inf_I\sqrt{4-x^2}$ and $$ \tau_k^*=(2\ln n)^{\frac{1}{2}}(nS(I)m_k^*-(32\ln n)^{\frac{1}{2}})/4+(5/8)\ln(2\ln n),$$ then 
		 the number of the rescaling point process
		$\{\tau_k^*\}$  falling in $[x,+\infty)$ tends to a Poisson random variable with mean $e^{c_2-x}$ for any fixed $x$. Here, the constant $c_2=c_0+M_0(I)$ depending on $I$, where $c_0=\frac1{12}\ln 2+3\zeta'(-1)$ and $M_0(I)=(3/2)\ln(4-a^2)-\ln(4|a|)$ if  $a+b<0,$ $M_0(I)=(3/2)\ln(4-b^2)-\ln(4|b|)$ if $a+b>0$ and $M_0(I)=(3/2)\ln(4-a^2)-\ln(2|a|)$ if $a+b=0$.	This further implies that, for any bounded interval $ I_1\subset\mathbb{R}$, the limiting density for the $k$-th largest gap falling in $I$ is given by  the Gumbel distribution,
		$$\lim_{n\to+\infty}\mathbb P(\tau_k^*\in I_1)=\int_{I_1}\frac{e^{k(c_2-x)}}{(k-1)!} e^{-e^{c_2-x}}dx.$$  In particular, the limiting density for the largest gap $\tau_1^*$ is
		$$e^{c_2-x} e^{-e^{c_2-x}}.$$
		
	\end{thm}
	
	Note that the constant $M(I)$ depends on whether or not $I$ is a symmetric subinterval about the origin, as the semicircle law is symmetric around the origin.
		
	The starting point to prove Theorem \ref{thm2} is  the observation  \eqref{21} in Lemma \ref{another}, which is another rescaling limit regarding the gap probability  for CUE.  Similar to the CUE case, we still need to prove the upper bound \eqref{33} and lower bound \eqref{34}.  And another key ingredient to prove the GUE case is the comparisons of the kernels and the Fredholm determinants between CUE and GUE in the proofs of Lemma \ref{lem23} and Lemma \ref{lem25}.

	As a final remark,   the determinantal
structures of the CUE and GUE play a crucial role in the analyses presented in this paper and in \cite{BB}. It is worth
noting that even the decay orders of the largest gaps for other ensembles,
such as the COE and CSE which have Pfaffian structures, remain unknown. In fact, let $m_{\beta}$ be the largest gaps of C$\beta$E where $\beta>0$.  Indicated by  the large gap probability of Theorem 5 in \cite{VV}, we propose the following conjecture: $$nm_{\beta}/\sqrt{\frac{64}\beta \ln n}\buildrel\hbox{\footnotesize $L^p$}\over\longrightarrow 1,$$
for any $p>0$ as $n\to+\infty$. 
If  this conjecture can be proven, then one can further try to find the rescaling limits of the largest gaps of C$\beta$E, which are expected to follow Gumbel distributions as well.  


	\section{A criterion for the Poisson convergence}One of the key ingredients is the following lemma for general decreasing point processes.
	\begin{lem}\label{lem1}Let $\chi^{(n)}=\sum\limits_{k=1}^{k_n}\delta_{\tau_k^{(n)}} $ be a sequence of point processes
		on $ \mathbb{R}$ such that the sequence $\tau_k^{(n)}\ (1\leq k\leq k_n)$ is decreasing for every fixed $n$, $f\in C^2(\mathbb{R})$ satisfies $f(x)>0,\ f'(x)<0,\ f''(x)>0$ for $x\in\mathbb{R}$ and $\lim\limits_{x\to+\infty}f'(x)=0.$ Assume that for every positive integer $k$ and $x_1,\cdots,x_k\in \mathbb{R}$, we have\begin{align}\label{43}&\lim_{n\to+\infty}\mathbb{E}\sum_{i_1,\cdots,i_{k}\ \text{all distinct}}\prod_{j=1}^k(\tau_{i_j}^{(n)}-x_j)_+=\prod_{j=1}^kf(x_j).
		\end{align}
		Then for $A=(x,+\infty)$ or $A=[x,+\infty)$, we have the convergence \begin{align}\label{44}{\chi}^{(n)}(A)\buildrel\hbox{\footnotesize law}\over\longrightarrow{\chi}(A),
		\end{align}where ${\chi}(A) $ is a Poisson random variable with mean $-f'(x).$
		Furthermore, for any bounded interval $ I\subset\mathbb{R}$, we have the limiting distribution,
		$$\lim_{n\to+\infty}\mathbb P(\tau_k^{(n)}\in I)=\int_I\frac{f''(x)(-f'(x))^{k-1}}{(k-1)!} e^{f'(x)}dx.$$ Here, we denote $\tau_k^{(n)}=-\infty $ for $k>k_n.$\end{lem}
	
	\begin{proof}For $a<b,\ x\in \mathbb{R}$, we simply have $$(b-a)\mathds{1}_{\{x\geq b\}}\leq(x-a)_+-(x-b)_+\leq (b-a)\mathds 1_{\{x>a\}},$$then for $a_1<a_{-1}$, we have \begin{align*}(a_{-1}-a_1)^k\prod_{j=1}^k\mathds 1_{\{\tau^{(n)}_{i_j}\geq a_{-1}\}}\leq\prod_{j=1}^k((\tau^{(n)}_{i_j}-a_{1})_+-(\tau^{(n)}_{i_j}-a_{-1})_+)
			\\ =\sum_{\varepsilon_1,\cdots,\varepsilon_{k}\in\{\pm1\}}\prod_{j=1}^k\varepsilon_{j}(\tau^{(n)}_{i_j}-a_{\varepsilon_{j}})_+
			\leq(a_{-1}-a_1)^k\prod_{j=1}^k\mathds 1_{\{\tau^{(n)}_{i_j}>a_{1}\}}.\end{align*}We denote $$ \rho^{(n,k)}=\sum_{i_1,\cdots,i_{k}\ \text{all distinct}}\delta_{\tau_{i_1}^{(n)},\cdots,\tau_{i_k}^{(n)}},$$  then we have $$\rho^{(n,k)}(A^k)=\dfrac{({\chi}^{(n)}(A))!}{({\chi}^{(n)}
			(A)-k)!}$$ for every interval $A\subset\mathbb{R}$. By taking summation over distinct points, we have \begin{align*}&(a_{-1}-a_1)^k\rho^{(n,k)}([a_{-1},+\infty)^k)
			\\ \leq&\sum_{i_1,\cdots,i_{k}\ \text{all distinct}}\sum_{\varepsilon_1,\cdots,\varepsilon_{k}\in\{\pm1\}}\prod_{j=1}^k\varepsilon_{j}(\tau^{(n)}_{i_j}-a_{\varepsilon_{j}})_+
			\\ \leq&(a_{-1}-a_1)^k\rho^{(n,k)}((a_{1},+\infty)^k).\end{align*}Using \eqref{43}, taking expectation and the limit, we have\begin{align*}&(a_{-1}-a_1)^k\limsup_{n\to+\infty}\mathbb{E}\rho^{(n,k)}([a_{-1},+\infty)^k)\\ \leq&\sum_{\varepsilon_1,\cdots,\varepsilon_{k}\in\{\pm1\}}\lim_{n\to+\infty}\mathbb{E}\sum_{i_1,\cdots,i_{k}\ \text{all distinct}}\prod_{j=1}^k\varepsilon_{j}(\tau^{(n)}_{i_j}-a_{\varepsilon_{j}})_+
			\\=&\sum_{\varepsilon_1,\cdots,\varepsilon_{k}\in\{\pm1\}}\prod_{j=1}^k\left(\varepsilon_{j}f(a_{\varepsilon_{j}})\right)=
			\left(f(a_{1})-f(a_{-1})\right)^k\\ \leq&(a_{-1}-a_1)^k\liminf_{n\to+\infty}\mathbb{E}\rho^{(n,k)}((a_{1},+\infty)^k).
		\end{align*}For every $ x\in\mathbb{R}$ and $ \delta>0,$ taking $ (a_1,a_{-1})=(x,x+\delta)$ and $ (a_1,a_{-1})=(x-\delta,x)$, we will  have\begin{align*}((f(x)-f(x+\delta))/\delta)^k\leq\liminf_{n\to+\infty}
			\mathbb{E}\rho^{(n,k)}((x,+\infty)^k)\\ \leq\limsup_{n\to+\infty}\rho^{(n,k)}([x,+\infty)^k)\leq
			((f(x-\delta)-f(x))/\delta)^k.
		\end{align*}Letting $ \delta\to 0+$ and using $\rho^{(n,k)}((x,+\infty)^k)\leq\rho^{(n,k)}([x,+\infty)^k) $, we have the following  convergence of the
		factorial moments, \begin{align*}\lim_{n\to+\infty}
			\mathbb{E}\dfrac{({\chi}^{(n)}(A))!}{({\chi}^{(n)}
				(A)-k)!}=\lim_{n\to+\infty}
			\mathbb{E}\rho^{(n,k)}(A^k)=(-f'(x))^k,
		\end{align*}where $A=(x,+\infty)$ or $A=[x,+\infty), $ 
		which implies the convergence of \eqref{44}.
		
		Now for every $k\geq 0,\ k\in\mathbb{Z}$, we have\begin{align*}\lim_{n\to+\infty}\mathbb{P}(\chi^{(n)}(A)=k)=\mathbb{P}(\chi(A)=k)=\left(-f'(x)\right)^ke^{f'(x)}/k!.
		\end{align*}Therefore, for $A=(x,+\infty)$ or $A=[x,+\infty)$, we have\begin{align}\label{45}\lim_{n\to+\infty}\mathbb{P}(\tau_k^{(n)}\in A)=\lim_{n\to+\infty}\mathbb{P}(\chi^{(n)}(A)\geq k)=\mathbb{P}(\chi(A)\geq k)=\varphi_k\left(-f'(x)\right),
		\end{align} where \begin{align*}\varphi_k\left(\lambda\right)=1-\sum_{j=0}^{k-1}\frac{\lambda^j}{j!}e^{-\lambda},
		\end{align*} thus \begin{align*}\ \varphi_k(0)=0,\ \varphi_k'\left(\lambda\right)=-\sum_{j=1}^{k-1}\frac{\lambda^{j-1}}{(j-1)!}e^{-\lambda}+
			\sum_{j=0}^{k-1}\frac{\lambda^j}{j!}e^{-\lambda}=\frac{\lambda^{k-1}}{(k-1)!}e^{-\lambda}\end{align*}and
		\begin{align*}\varphi_k\left(\lambda\right)=\int_0^{\lambda}\varphi_k'\left(s\right)ds=
			\int_0^{\lambda}\frac{s^{k-1}}{(k-1)!}e^{-s}ds.\end{align*}Changing variables $s=-f'(x)$, we have
		\begin{align}\label{46}\varphi_k\left(-f'(a)\right)=
			\int_0^{-f'(a)}\frac{s^{k-1}}{(k-1)!}e^{-s}ds=
			\int_a^{+\infty}\frac{f''(x)(-f'(x))^{k-1}}{(k-1)!}e^{f'(x)}dx\end{align}for every $a\in\mathbb{R}.$ Now for any bounded interval $ I\subset\mathbb{R}$, we can write $I=(a,b)$ or $I=(a,b]$ or $I=[a,b)$ or $I=[a,b]$ where $a<b$, thus $I=A_1\setminus A_2$ with $A_1=(a,+\infty)$ or $A_1=[a,+\infty)$ and $A_2=(b,+\infty)$ or $A_2=[b,+\infty)$, and by \eqref{45} and \eqref{46} we have
		\begin{align*}\lim_{n\to+\infty}\mathbb{P}(\tau_k^{(n)}\in I)=\lim_{n\to+\infty}\mathbb{P}(\tau_k^{(n)}\in A_1)-
			\lim_{n\to+\infty}\mathbb{P}(\tau_k^{(n)}\in A_2)\\=\varphi_k\left(-f'(a)\right)-\varphi_k\left(-f'(b)\right)=
			\int_a^{b}\frac{f''(x)(-f'(x))^{k-1}}{(k-1)!}e^{f'(x)}dx.
		\end{align*}This completes the proof.\end{proof}

	\section{The CUE case}
	\subsection{A rescaling limit}
	For CUE, the gap probability of having no eigenvalue in an arc of size $2\alpha$ is
	equal to the Toeplitz determinant $$D_n(\alpha)=\det_{1\leq j,k\leq n}\left(\frac{1}{2\pi}\int_{\alpha}^{2\pi-\alpha}e^{i(j-k)\theta}d\theta\right).$$ All the asymptotics we need  are direct consequences of
	the precise analysis of $D_n(\alpha) $ given by Deift et al. \cite{DIKZ}. More precisely they proved that for some sufficiently large $s_0$ and any $ \varepsilon> 0$,
	uniformly in $s_0/n<\alpha<\pi-\varepsilon,$ one has  \begin{align}\label{1}\ln D_n(\alpha)=n^2\ln\cos\frac{\alpha}{2}-\frac{1}{4}\ln\left(n\sin\frac{\alpha}{2}\right)+c_0+O\left(\frac{1}{n\sin(\alpha/2)}\right),
	\end{align}
	here $c_0=\frac 1{12}\ln 2+3\zeta'(-1)$ where $\zeta(x)$ is the Riemann zeta function. 
	
	We denote\begin{equation}\label{fx}F_n(x)=\frac{8x+3\ln(2\ln n)}{2n(2\ln n)^{\frac{1}{2}}}+\frac{(32\ln n)^{\frac{1}{2}}}{n},\end{equation} then we have $$m_k=F_n(\tau_k),$$
	where $m_k$ and $\tau_k$ are as defined in Theorem \ref{thm1}.
	
	From the definition of $F_n(x) $, we have\begin{align}\label{10}\tau_k-x=(F_n(\tau_k)-F_n(x))(n/4)(2\ln n)^{\frac{1}{2}}=(m_k-F_n(x))(n/4)(2\ln n)^{\frac{1}{2}}, \end{align} and for every fixed $x,$ we have \begin{align}\label{11}\lim\limits_{n\to+\infty}\dfrac{n F_n(x)}{(32\ln n)^{\frac{1}{2}}}=1,\ \lim\limits_{n\to+\infty}n F_n(x)=+\infty,\ \lim\limits_{n\to+\infty}n^{\gamma} F_n(x)=0,\ \forall\ \gamma<1.\end{align}For every fixed $ \alpha\in(0,\pi),$ by \eqref{1} we have\begin{align}\label{12}&\lim_{n\to+\infty}(n/4)(2\ln n)^{\frac{1}{2}}D_n(\alpha)=0.
	\end{align}
	Another important consequence of \eqref{1} is the following rescaling limit
	\begin{lem}\label{obs}\begin{align}\label{2}\lim_{n\to+\infty}n(2\ln n)^{\frac{1}{2}} D_n(F_n(x)/2)=e^{c_0-x}.\end{align}\end{lem}
	\begin{proof}Let $\alpha_n=F_n(x)/2,$ then by \eqref{11} we have $\alpha_n\to 0,\ n\alpha_n\to+\infty$ as $ n\to+\infty,$ thus $s_0/n<\alpha_n<\pi-\varepsilon $ for $n$ sufficiently large, and\begin{equation}\label{dds}\lim\limits_{n\to+\infty}\frac{1}{n\sin(\alpha_n/2)}=\lim\limits_{n\to+\infty}\frac{2}{n\alpha_n}
			\lim\limits_{n\to+\infty}\frac{\alpha_n/2}{\sin(\alpha_n/2)}=0.\end{equation}Thus, by \eqref{1} we have\begin{equation}\label{3}\lim\limits_{n\to+\infty}\left(\ln D_n(\alpha_n)-n^2\ln\cos\frac{\alpha_n}{2}+\frac{1}{4}\ln\left(n\sin\frac{\alpha_n}{2}\right)-c_0\right)=0.
		\end{equation}By \eqref{11} we have\begin{align*}\lim\limits_{n\to+\infty}\frac{(2\ln n)^{\frac{1}{2}}}{n\sin(\alpha_n/2)}&=\lim\limits_{n\to+\infty}\frac{(2\ln n)^{\frac{1}{2}}}{n\alpha_n/2}
			\lim\limits_{n\to+\infty}\frac{\alpha_n/2}{\sin(\alpha_n/2)}=\lim\limits_{n\to+\infty}\frac{(2\ln n)^{\frac{1}{2}}}{n\alpha_n/2}\\&=\lim\limits_{n\to+\infty}\frac{(2\ln n)^{\frac{1}{2}}}{nF_n(x)/4}=\lim\limits_{n\to+\infty}\frac{(32\ln n)^{\frac{1}{2}}}{nF_n(x)}=1,\end{align*} and thus we have \begin{align}\label{4}\lim\limits_{n\to+\infty}\left(\frac{1}{8}\ln (2\ln n)-\frac{1}{4}\ln\left(n\sin\frac{\alpha_n}{2}\right)\right)=0.
		\end{align}By \eqref{11} and  the Taylor expansion $\ln \cos y=-y^2/2+O(y^4)$ as $y\to 0$, we have \begin{align*}n^2\ln\cos\frac{\alpha_n}{2}+\frac{n^2\alpha_n^2}{8}=n^2O(\alpha_n^4)=n^2O(F_n^4(x))=O\left((n^{1/2} F_n(x))^4\right)\to 0,\end{align*} and\begin{align*}\frac{n^2\alpha_n^2}{8}=\frac{n^2F_n^2(x)}{32}&=\frac{32\ln n}{32}+\frac{8x+3\ln(2\ln n)}{(2\ln n)^{\frac{1}{2}}}\frac{(32\ln n)^{\frac{1}{2}}}{32}+\frac{(8x+3\ln(2\ln n))^2}{32\cdot4\cdot(2\ln n)}\\&=\ln n+\frac{8x+3\ln(2\ln n)}{8}+o(1)\end{align*}as $n\to+\infty$, which implies\begin{align}\label{5}\lim\limits_{n\to+\infty}\left(n^2\ln\cos\frac{\alpha_n}{2}+\ln n+x+\frac{3\ln(2\ln n)}{8}\right)=0.
		\end{align}By \eqref{3}\eqref{4}\eqref{5}, we have\begin{align*}\lim\limits_{n\to+\infty}\left(\ln D_n(\alpha_n)+\ln n+x+\frac{\ln(2\ln n)}{2}-c_0\right)=0,\end{align*}and thus we have \begin{align*}\lim\limits_{n\to+\infty}\ln\left(n(2\ln n)^{\frac{1}{2}} D_n(\alpha_n)\right)= c_0-x.\end{align*}As $\alpha_n=F_n(x)/2$, we finally have\begin{align*}&\lim_{n\to+\infty}n(2\ln n)^{\frac{1}{2}} D_n(F_n(x)/2) =e^{ c_0-x},\end{align*}which completes the proof of \eqref{2}.\end{proof}

	\subsection{The strategy to prove Theorem \ref{thm1}}Now we take $c_1=c_0+\ln(\pi/2),\ f(x)=e^{c_1-x}=(2\pi)e^{c_0-x}/4,$ then we have $-f'(x)=f''(x)=e^{c_1-x}. $ Thanks to Lemma \ref{lem1}, for every positive integer $k$, $x_1,\cdots,x_k\in \mathbb{R}$,  and $\tau_{j}$ is as defined in Theorem \ref{thm1},  if we can prove the following convergence\begin{align}\label{41}&\lim_{n\to+\infty}\mathbb{E}\sum_{i_1,\cdots,i_{k}\ \text{all distinct}}\prod_{j=1}^k(\tau_{i_j}-x_j)_+
		=(2\pi)^k\prod_{j=1}^k\left(e^{c_0-x_j}/4\right),
	\end{align}then Theorem \ref{thm1} will be proved.
	
	We need to introduce some notations. For a set $A\subset\mathbb{R}$, we denote $A(\text{mod}\ 2\pi):=\{x+2\pi k|x\in A,k\in \mathbb{Z}\}\cap[0,2\pi).$ Then $I(x,a):=[x,x+a](\text{mod}\ 2\pi)$ is an arc of size $a$ for $a\in(0,2\pi)$. For $0< \theta_1 < \cdots < \theta_n < 2\pi$ and $ \theta_{k+n}=\theta_k+2\pi,$ denote $J_k(a):=\{x\in [0,2\pi)|I(x,a)\subset(\theta_k,\theta_{k+1})(\text{mod}\ 2\pi)\}$ for $a\in(0,2\pi),\ 1\leq k\leq n,$ then we have $J_k(a)=(\theta_k,\theta_{k+1}-a)(\text{mod}\ 2\pi)$ for $\theta_{k+1}-\theta_k>a $ and $J_k(a)=\emptyset$ for $\theta_{k+1}-\theta_k\leq a, $ thus $J_k(a)$ is an arc of size $(\theta_{k+1}-\theta_k-a)_+,$ moreover, $J_k(a)\subset(\theta_k,\theta_{k+1})(\text{mod}\ 2\pi) $ and $J_k(a)\cap J_l(a)=\emptyset $ for $k\neq l.$ Now let the set \begin{align*}&\Sigma_k(a_1,\cdots,a_k):=\bigcup_{i_1,\cdots,i_{k}\ \text{all distinct}}\prod_{j=1}^kJ_{i_j}(a_j)\subset[0,2\pi)^k,
	\end{align*}then this is in fact a disjoint union and\begin{align*}|\Sigma_k(a_1,\cdots,a_k)|&=\sum_{i_1,\cdots,i_{k}\ \text{all distinct}}\prod_{j=1}^k(\theta_{i_j+1}-\theta_{i_j}-a_j)_+\\ &=\sum_{i_1,\cdots,i_{k}\ \text{all distinct}}\prod_{j=1}^k(m_{i_j}-a_j)_+,
	\end{align*}here, we denote $|X|$ as the $k$-dimensional Lebesgue measure of a set $X\subset\mathbb{R}^k.$ By \eqref{11}, for every fixed $x_1,\cdots,x_k\in \mathbb{R}$, there exists $N_0>0$ such that $0<2s_0/n<F_n(x_j)<1<2\pi$ for $n>N_0,\ 1\leq j\leq k.$ Now we always assume $n>N_0.$ By \eqref{10}, we have \begin{align*}&\sum_{i_1,\cdots,i_{k}\ \text{all distinct}}\prod_{j=1}^k(\tau_{i_j}-x_j)_+
		\\=&(n/4)^k(2\ln n)^{\frac{k}{2}}\sum_{i_1,\cdots,i_{k}\ \text{all distinct}}\prod_{j=1}^k(m_{i_j}-F_n(x_j))_+\\=&(n/4)^k(2\ln n)^{\frac{k}{2}}|\Sigma_k(F_n(x_1),\cdots,F_n(x_k))|.
	\end{align*}For fixed $x_1,\cdots,x_k\in \mathbb{R}$ and  $y_1,\cdots,y_k\in [0,2\pi)$, let's denote  \begin{align*} &\phi_{k,n}(y_1,\cdots,y_k):= (n/4)^k(2\ln n)^{\frac{k}{2}}\times \\ &\mathbb{P}\Big((y_1,\cdots,y_k)\in\Sigma_k(F_n(x_1),\cdots,F_n(x_k))\Big),\end{align*} then we can rewrite \begin{align*}&\mathbb{E}\sum_{i_1,\cdots,i_{k}\ \text{all distinct}}\prod_{j=1}^k(\tau_{i_j}-x_j)_+
		\\=&\mathbb{E}(n/4)^k(2\ln n)^{\frac{k}{2}}|\Sigma_k(F_n(x_1),\cdots,F_n(x_k))|\\=&\int_{[0,2\pi)^k}\phi_{k,n}(y_1,\cdots,y_k)dy_1\cdots dy_k.
	\end{align*}Hence, \eqref{41} will be the direct consequence of the following two inequalities and the
	dominated convergence theorem: we  will prove the upper bound \begin{align}\label{13}&\limsup_{n\to+\infty}\sup_{y_1,\cdots,y_k\in [0,2\pi)}\phi_{k,n}(y_1,\cdots,y_k)\leq \prod_{j=1}^k\left(e^{c_0-x_j}/4\right);
	\end{align}  and if all $y_k$'s are distinct, then we will prove the lower bound \begin{align}\label{14}&\liminf_{n\to+\infty}\phi_{k,n}(y_1,\cdots,y_k)\geq \prod_{j=1}^k\left(e^{c_0-x_j}/4\right).
	\end{align}
	\subsection{The proof of Theorem \ref{thm1}} Let's prove Theorem \ref{thm1}.
	\subsubsection{An equivalent condition}
	We first need  the following equivalent condition for a point $(y_1,\cdots,y_k)$ in  the set $\Sigma_k(a_1,\cdots,a_k)$. \begin{lem}\label{lem2}$(y_1,\cdots,y_k)\in\Sigma_k(a_1,\cdots,a_k) $ is equivalent to the following conditions: (i) $I(y_l,a_l)\cap I(y_j,a_j)=\emptyset $ for $1\leq l<j\leq k,$ and (ii) $ \theta_l\not\in I(y_j,a_j),$ for $1\leq j\leq k,\ 1\leq l\leq n$, and (iii) $\{ \theta_1,\cdots ,\theta_n \}\cap(y_p,y_q)\neq\emptyset,\ \{ \theta_1,\cdots ,\theta_n \}\setminus[y_p,y_q]\neq\emptyset$ for every $p,q\in\{1,\cdots,k\}$ such that $y_p<y_q.$\end{lem}\begin{proof}If $(y_1,\cdots,y_k)\in\Sigma_k(a_1,\cdots,a_k) $, then we can find $i_1,\cdots,i_{k}\in\{1,\cdots,n\}$ {all distinct} such that $y_j\in J_{i_j}(a_j),$ thus $I(y_j,a_j)\subset(\theta_{i_j},\theta_{i_j+1})(\text{mod}\ 2\pi),$ and $I(y_l,a_l)\cap I(y_j,a_j)\subseteq(\theta_{i_l},\theta_{i_l+1})(\text{mod}\ 2\pi)\cap(\theta_{i_j},\theta_{i_j+1})(\text{mod}\ 2\pi)=\emptyset $ for $1\leq l<j\leq k,$ since $i_l\neq i_j,$ which gives (i).
		
		Since $0< \theta_1 < \cdots < \theta_n < 2\pi$, we have $\theta_l\not\in (\theta_{j},\theta_{j+1})(\text{mod}\ 2\pi) $ for $1\leq j,l\leq n$. Thus for $1\leq j\leq k,\ 1\leq l\leq n$, we have $\theta_l\not\in (\theta_{i_j},\theta_{i_j+1})(\text{mod}\ 2\pi) $ and $I(y_j,a_j)\subset(\theta_{i_j},\theta_{i_j+1})(\text{mod}\ 2\pi),$ which implies (ii) $ \theta_l\not\in I(y_j,a_j).$
		
		For every $p,q\in\{1,\cdots,k\},$ such that $y_p<y_q,$ we have $i_p\neq i_q.$ If $i_p,i_q\neq n$ then we have $y_p\in I(y_p,a_p)\subset(\theta_{i_p},\theta_{i_p+1})(\text{mod}\ 2\pi)=(\theta_{i_p},\theta_{i_p+1}),$ and similarly $y_q\in(\theta_{i_q},\theta_{i_q+1})$. Therefore, $\theta_{i_p}<y_p<y_q<\theta_{i_q+1} $ and $i_p<{i_q+1}, $ since $i_p,i_q\in \mathbb{Z},$ we have $i_p\leq i_q$, since $i_p\neq i_q,$ we have $i_p< i_q$ and $i_p+1\leq i_q.$ Thus $0<\theta_{i_p}<y_p<\theta_{i_p+1}\leq \theta_{i_q}<y_q $ and $\theta_{i_p+1}\in(y_p,y_q), $ $\theta_{i_p}\not\in[y_p,y_q], $ which implies (iii).
		
		If $i_p\neq i_q= n$, then we have $y_p\in (\theta_{i_p},\theta_{i_p+1})$ and $y_q\in(\theta_{i_q},\theta_{i_q+1})(\text{mod}\ 2\pi)=(\theta_{n},2\pi)\cup[0,\theta_1)$.  Thus $\theta_1\leq \theta_{i_p}<y_p<y_q, $ which implies $y_q\not\in[0,\theta_1)$ and $y_q\in(\theta_{n},2\pi).$ Now we have $i_p<n,$ $0<\theta_{i_p}<y_p<\theta_{i_p+1}\leq \theta_{n}, $ and $\theta_{i_p+1}\in(y_p,y_q), $ $\theta_{i_p}\not\in[y_p,y_q], $ which implies (iii).
		
		If $i_p=n\neq i_q$,  then we have $y_p\in (\theta_{n},2\pi)\cup[0,\theta_1)$ and $y_q\in(\theta_{i_q},\theta_{i_q+1}),\ i_q<n.$ Thus $y_p<y_q<\theta_{i_q+1}\leq\theta_{n}, $ which implies $y_p\not\in(\theta_{n},2\pi)$ and $y_p\in[0,\theta_1).$ Now we have  $y_p<\theta_{1}\leq \theta_{i_q}<y_q<\theta_{i_q+1}<\pi $ and $\theta_{1}\in(y_p,y_q), $ $\theta_{i_q+1}\not\in[y_p,y_q], $ which also implies (iii). Now we finish the proof of the first part.
		
		Conversely if (i)(ii)(iii) are true, by (ii) there exists a unique $i_{j}\in\{1,\cdots,n\}$ such that $I(y_j,a_j)\subset(\theta_{i_j},\theta_{i_j+1})(\text{mod}\ 2\pi), $ by (i) we know that all $y_k$'s are distinct.
		
		If $i_p=i_q$ for some $p,q\in\{1,\cdots,k\}$ with $p\neq q,$ we can assume $y_p<y_q.$ If $i_p=i_q<n$, then we have $y_p\in (\theta_{i_p},\theta_{i_p+1})$ and $y_q\in(\theta_{i_q},\theta_{i_q+1})=(\theta_{i_p},\theta_{i_p+1}),$ thus $\theta_{i_p}<y_p<y_q<\theta_{i_p+1}, $ and $\{ \theta_1,\cdots ,\theta_n \}\cap(y_p,y_q)=\emptyset,$ which contradicts (iii).
		
		If $i_p=i_q=n$, then we have $y_p\in (\theta_{i_p},\theta_{i_p+1})(\text{mod}\ 2\pi)=(\theta_{n},2\pi)\cup[0,\theta_1)$ and \\ $y_q\in(\theta_{i_q},\theta_{i_q+1})(\text{mod}\ 2\pi)=(\theta_{n},2\pi)\cup[0,\theta_1).$ Thus,  if $y_q<\theta_1$, then $ y_p<y_q<\theta_{1}$ and $\{ \theta_1,\cdots ,\theta_n \}\cap(y_p,y_q)=\emptyset;$ if $y_p>\theta_n$, then $\theta_{n}< y_p<y_q$ and $\{ \theta_1,\cdots ,\theta_n \}\cap(y_p,y_q)=\emptyset;$ if $y_p\leq \theta_n,\ y_q\geq\theta_{1}$, then $y_p\in [0,\theta_1)$,$y_q\in(\theta_{n},2\pi),$ and $\{ \theta_1,\cdots ,\theta_n \}\setminus[y_p,y_q]=\emptyset.$ All the 3 cases contradict (iii).
		
		Therefore, we must have $i_p\neq i_q$ for every $p,q\in\{1,\cdots,k\},\ p\neq q,$ i.e., $i_1,\cdots,i_{k}\in\{1,\cdots,n\}$ are {all distinct}, and $I(y_j,a_j)\subset(\theta_{i_j},\theta_{i_j+1})(\text{mod}\ 2\pi),\ y_j\in J_{i_j}(a_j),$ which implies $(y_1,\cdots,y_k)\in\Sigma_k(a_1,\cdots,a_k). $ This completes the proof.\end{proof}
	
	\subsubsection{Upper bound}The proof of the upper bound is based on the following negative correlation of the vacuum events for the determinantal point processes. We refer to Lemma 3.8 in \cite{BB} for its proof.
	
	\begin{lem}\label{lem3}Let $\xi^{(n)} $ be the
		point process associated to the eigenvalues of Haar-distributed unitary matrix
		(resp., an element of the GUE). Let $I_1$ and $I_2$ be compact disjoint subsets
		of $[0, 2\pi)$ (resp., $ \mathbb{R}$). Then\begin{align}\label{15}&\mathbb{P}(\xi^{(n)}(I_1\cup I_2)=0)\leq \mathbb{P}(\xi^{(n)}(I_1)=0) \mathbb{P}(\xi^{(n)}(I_2)=0).
	\end{align}\end{lem}By monotone convergence theorem, we have $$ \mathbb{P}(\xi^{(n)}(\cup_{j=1}^{+\infty}J_j)=0)=\lim\limits_{k\to+\infty}\mathbb{P}(\xi^{(n)}(\cup_{j=1}^{k}J_j)=0),$$ thus \eqref{15} is also true if $I_1$ and $I_2$ are disjoint $F_{\sigma}$ subsets (i.e. $I_k=\cup_{j=1}^{+\infty}I_{k,j}$ and $I_{k,j} $ are compact), especially the subsets in the form of $ (a,b)(\text{mod}\ 2\pi)$ or $ [a,b](\text{mod}\ 2\pi)$. By induction, for disjoint $F_{\sigma}$ subsets $I_1,\cdots, I_k$, we also have \begin{align}\label{16}&\mathbb{P}(\xi^{(n)}(\cup_{j=1}^{k}I_j)=0)\leq \prod_{j=1}^{k}\mathbb{P}(\xi^{(n)}(I_j)=0) .
	\end{align}By definition of $ D_n(\alpha),$ for $a\in(0,2\pi),\ x\in\mathbb{R},$ we have \begin{align}\label{17}&\mathbb{P}(\xi^{(n)}(I(x,a))=0)=D_n(a/2) .
	\end{align} We consider the point process $$ \xi^{(n)}=\sum_{i=1}^n\delta_{\theta_i}.$$ For fixed $x_1,\cdots,x_k\in \mathbb{R},\ n>N_0,$ let's denote $$A_n:=\Big\{(y_1,\cdots,y_k)\in[0,2\pi)^k\big |I(y_i,F_n(x_i))\cap I(y_j,F_n(x_j))=\emptyset,\ \forall\ 1\leq i<j\leq k\Big \}.$$ 
	If $ (y_1,\cdots,y_k)\in A_n,$ then all $y_k$'s are distinct, let \begin{align}\label{18}I_{n,k}= \cup_{j=1}^kI(y_j, F_n(x_j)), J_{n,k, j}=(z_j,z_{j+1})(\text{mod}\ 2\pi),\ 1\leq j\leq k,\end{align} here, $z_j (1\leq j\leq k)$ is the increasing rearrangement of $y_j (1\leq j\leq k)$ and $z_{k+1}=z_1+2\pi.$ Then $I_{n,k}$ is a disjoint union and by Lemma \ref{lem2} we have\begin{align}\label{19}&\phi_{k,n}(y_1,\cdots,y_k)=(n/4)^k(2\ln n)^{\frac{k}{2}}\times\\&\nonumber \mathbb{P}(\xi^{(n)}(I_{n,k})=0,\ \xi^{(n)}(J_{n,k,j})>0,\ \forall\ 1\leq j\leq k).
	\end{align} By \eqref{16} and \eqref{17} we have\begin{align*}\phi_{k,n}(y_1,\cdots,y_k)&\leq(n/4)^k(2\ln n)^{\frac{k}{2}}\mathbb{P}(\xi^{(n)}(I_{n,k})=0)\\ &\leq(n/4)^k(2\ln n)^{\frac{k}{2}}\prod_{j=1}^k\mathbb{P}(\xi^{(n)}(I(y_j, F_n(x_j)))=0)\\ &=(n/4)^k(2\ln n)^{\frac{k}{2}}\prod_{j=1}^kD_n(F_n(x_j)/2).
	\end{align*}For  $ (y_1,\cdots,y_k)\in [0,2\pi)^k\setminus A_n$, by Lemma \ref{lem2},  we have\begin{align*}&\phi_{k,n}(y_1,\cdots,y_k)=0\leq(n/4)^k(2\ln n)^{\frac{k}{2}}\prod_{j=1}^kD_n(F_n(x_j)/2).
	\end{align*}Therefore,  by \eqref{2} we always have\begin{align*}&\sup_{y_1,\cdots,y_k\in [0,2\pi)}\phi_{k,n}(y_1,\cdots,y_k)\leq(n/4)^k(2\ln n)^{\frac{k}{2}}\prod_{j=1}^kD_n(F_n(x_j)/2)\\&=\prod_{j=1}^k(n(2\ln n)^{\frac{1}{2}}D_n(F_n(x_j)/2)/4)\to \prod_{j=1}^k\left(e^{c_0-x_j}/4\right), \,\,\, n\to+\infty, 
	\end{align*} which gives the upper bound \eqref{13}.
	\subsubsection{Lower bound}
	Now we consider the lower bound.
	
	If all $y_k$'s are distinct, let $z_j$ be the increasing rearrangement of $y_j$ and $z_{k+1}=z_1+2\pi$ as above. By  \eqref{11}, there further exists $N_1>N_0$ (depending only on $x_1,\cdots,x_k$ and  $ y_1,\cdots,y_k$) such that $0<2s_0/n<F_n(x_j)<\min\{z_{i+1}-z_i|1\leq i\leq k\}/2$ for $n>N_1,\ 1\leq j\leq k.$ Then we have $(y_1,\cdots,y_k)\in A_n$ for $n>N_1,$ and we can still use the notation \eqref{18} and formula \eqref{19} in this case. The proof of the lower bound is based on the following asymptotic splitting property,
	\begin{lem}\label{splt}
		\begin{align}\label{42}&\lim_{n\to+\infty}\mathbb{P}(\xi^{(n)}(I_{n,k})=0)/\prod_{j=1}^k\mathbb{P}(\xi^{(n)}(I(y_j, F_n(x_j)))=0)=1.
	\end{align}\end{lem}  For a nuclear operator $T$ in the form of $$Tf(x)=\int K(x,y)f(y)dy,$$ the Hilbert-Schmidt norm is given by\begin{align*}&|T|_2^2=\int\int |K(x,y)|^2dxdy
	\end{align*}and the trace is given by $$\Tr T=\int K(x,x) dx.$$ For a bounded operator $T$ on $L^2$-space, the operator norm is given by $$\|T\|=\sup\Big\{\|Tf\|_{L^2}|\ \|f\|_{L^2}=1\Big \}.$$ Let's  recall that the probability that a determinantal
	point process $\xi$ with kernel $K$ has no point in a measurable subset $A$ is given by the
	Fredholm determinant \cite{AGZ}\begin{align*}&\mathbb{P}(\xi(A)=0)=\det (\text{Id} -K_A).
	\end{align*}
	In the case of CUE,  the point process of eigenvalues $ \xi^{(n)}$ is a determinantal
	point process with kernel \cite{AGZ}, \begin{align}\label{kern}K_n(x,y):&=K^{CUE(n)}(x,y)=\dfrac{1}{2\pi}\dfrac{\sin (n(x-y)/2)}{\sin ((x-y)/2)}\\\nonumber&=\dfrac{1}{2\pi}\sum\limits_{k=0}^{n-1}e^{(k-(n-1)/2)i(x-y)}.\end{align} 
	Therefore,  the probability that  $\xi^{(n)}$ has no point  in a measurable subset $I$ is $$\mathbb{P}(\xi^{(n)}(I)=0)=\det (\text{Id} -\chi_I P_n\chi_I), $$ where  $P_n$ is the orthogonal projection from $L^2([0,2\pi))$ to the finite dimensional space $V_n:=\text{span}\{e^{i(k-(n-1)/2)x}|0\leq k<n,k\in\mathbb{Z}\}$ with kernel $K_n(x,y)$ in \eqref{kern}, and $\chi_I$ is the characteristic function supported on $I$. Assume $n>N_1$ and denote \begin{equation}\label{defi}A=\chi_{I_{n,k}} P_n\chi_{I_{n,k}},\ B=\sum\limits_{j=1}^kB_j,\ B_j=\chi_{I(y_j, F_n(x_j))} P_n\chi_{I(y_j, F_n(x_j))},\end{equation}then we have $$\mathbb{P}(\xi^{(n)}(I_{n,k})=0)=\det (\text{Id} -A); $$ since the support $I(y_j, F_n(x_j))$'s are disjoint, we also have\begin{align*}&\prod_{j=1}^k\mathbb{P}(\xi^{(n)}(I(y_j, F_n(x_j)))=0)=\prod_{j=1}^k\det (\text{Id} -B_j)=\det (\text{Id} -B).
	\end{align*} Now \eqref{42} is equivalent to\begin{align}\label{47}&\lim_{n\to+\infty}\det (\text{Id} -A)/\det (\text{Id} -B)=1.
	\end{align}By \eqref{1} we have \begin{align*}&\det (\text{Id} -B)=\prod_{j=1}^k\mathbb{P}(\xi^{(n)}(I(y_j, F_n(x_j)))=0)=\prod_{j=1}^kD_n(F_n(x_j)/2)>0,
	\end{align*}and thus $\det (\text{Id} -A)/\det (\text{Id} -B) $ is well defined. Since $P_n$ is a finite rank orthogonal projection operator, we know that $A,\ B$ are both finite rank symmetric operators. As $ \langle Af,f\rangle=\langle P_n\chi_{I_{n,k}}f,\chi_{I_{n,k}}f\rangle=\|P_n\chi_{I_{n,k}}f\|_{L^2}^2$, we have $$0\leq\langle Af,f\rangle=\|P_n\chi_{I_{n,k}}f\|_{L^2}^2\leq \|\chi_{I_{n,k}}f\|_{L^2}^2\leq \|f\|_{L^2}^2, $$ here, we use the $L^2$ inner product $$ \langle f,g\rangle:=\int_0^{2\pi}f(x)\overline{g(x)}dx,\ \|f\|_{L^2}^2=\langle f,f\rangle.$$
	Similarly, we have $0\leq\langle B_jf,f\rangle\leq \|\chi_{I(y_j, F_n(x_j))}f\|_{L^2}^2, $ and if $n>N_1$, then\\ $I(y_j, F_n(x_j))$'s are disjoint and\begin{align*}&0\leq\sum_{j=1}^k\langle B_jf,f\rangle=\langle Bf,f\rangle\leq\sum_{j=1}^k \|\chi_{I(y_j, F_n(x_j))}f\|_{L^2}^2\leq \|f\|_{L^2}^2.
	\end{align*}Therefore,  we can conclude that $A,\ B,\ \text{Id} -A,\ \text{Id} -B $ are all semi-positive definite. As $\det (\text{Id} -B)>0, $ then $\text{Id} -B $ is further positive definite, so is its inverse $(\text{Id}-B)^{-1}. $ Such results are also true for the GUE case in \S\ref{gu7}. 
	
	We will need the following general comparison inequalities regarding the Fredholm determinants which will be used in both CUE and GUE cases.\begin{lem}\label{lem15}Assume $A,B$ are finite rank symmetric operators on a Hilbert space, $\text{Id} -B $ is positive definite and $\text{Id} -A $ is semi-positive definite,  then we have $$1-|A-B|_2^2\|(\text{Id}-B)^{-1}\|^2\leq\exp(\Tr(A-B)(\text{Id}-B)^{-1})\det(\text{Id} -A)/\det(\text{Id}-B)\leq1 $$and $$|\Tr((A-B)(\text{Id}-B)^{-1})|\leq|\Tr(A-B)|+|A-B|_2|B|_2\|(\text{Id}-B)^{-1}\|. $$\end{lem}In the proof we need to use the following formulas     \cite{GGK} \begin{itemize}
		\item If $A,B$ are finite rank operators, then $\det(\text{Id} -A)\det(\text{Id} -B)=\det((\text{Id} -A)(\text{Id} -B))$ and  $|\Tr AB|\leq |A|_2|B|_2 $.
		\item If $A$ is a finite rank operator, $B$ is a bounded operator, then $\Tr AB=\Tr BA$ and  $|AB|_2\leq|A|_2\|B\|$.
	\end{itemize} If $B$ is a finite rank symmetric operator and $\text{Id} -B $ is positive definite, let $\{e_k\}$ be eigenfunctions forming a complete orthonormal basis  with $B e_k=\lambda_k(B) e_k$, then $ \lambda_k(B)\in\mathbb{R},\ \lambda_k(B)<1$. Now we have\begin{align*}&\det (\text{Id} -B)=\prod(1-\lambda_k(B)),\ \Tr B=\sum\lambda_k(B).
	\end{align*}We can also define $(\text{Id} -B)^p $ for every $p\in\mathbb{R}$ as\begin{align*}&(\text{Id} -B)^pf=\sum(1-\lambda_k(B))^p\langle f,e_k\rangle e_k=f+\sum((1-\lambda_k(B))^p-1)\langle f,e_k\rangle e_k.
	\end{align*}Then $(\text{Id} -B)^p $ is also positive definite, $(\text{Id} -B)^p(\text{Id} -B)^q=(\text{Id} -B)^{p+q} $ and $\det(\text{Id} -B)^p=(\det(\text{Id} -B))^p.$ Moreover, for $p<0$, we have $\|(\text{Id}-B)^{p}\|=(1-\lambda_1(B))^{p} $ where $\lambda_1(B) $ is the largest eigenvalue of $B$.\begin{proof}Since $\text{Id} -B $ is positive definite, so is its inverse $(\text{Id}-B)^{-1}$ and $(\text{Id}-B)^{-1}$ has a positive square root $(\text{Id}-B)^{-1/2}$. Moreover, $\|(\text{Id}-B)^{-1/2}\|^2=\|(\text{Id}-B)^{-1}\|=(1-\lambda_1(B))^{-1}, $ where $\lambda_1(B) $ is the largest eigenvalue of $B$ and $\lambda_1(B)<1 $.  We also have $(\det(\text{Id}-B)^{-1/2})^2=\det(\text{Id}-B)^{-1}=(\det(\text{Id}-B))^{-1} $. Since $\text{Id} -A $ is semi-positive definite, so is $A_1:=(\text{Id} -B)^{-1/2}(\text{Id} -A)(\text{Id} -B)^{-1/2},$ and $\det A_1= \det(\text{Id} -A)/\det(\text{Id}-B).$ Let $B_1:=(\text{Id} -B)^{-1/2}(A-B)(\text{Id} -B)^{-1/2},$ then we have $A_1+B_1=\text{Id},$ $B_1$ is a finite rank symmetric operator, $\Tr B_1=\Tr(A-B)(\text{Id}-B)^{-1},$ and its eigenvalues $ \lambda_j(B_1)$ are real. Since $A_1=\text{Id}-B_1$ is semi-positive definite, we have $ \lambda_j(B_1)\leq 1$ and   $ \det A_1=\det (\text{Id}-B_1)=\prod_j(1-\lambda_j(B_1))$. Now we can rewrite \begin{align}\label{48}&\exp(\Tr(A-B)(\text{Id}-B)^{-1})\det(\text{Id} -A)/\det(\text{Id}-B)\\\nonumber=&\exp(\Tr B_1)\det A_1\\ \nonumber=&\exp(\sum_j\lambda_j(B_1))\prod_j(1-\lambda_j(B_1))=\prod_j(e^{\lambda_j(B_1)}(1-\lambda_j(B_1))).
		\end{align}Since $ e^{\lambda}(1-\lambda)\leq 1$ and $1+\lambda\leq e^{\lambda},$ we have $(1+\lambda)_+\leq e^{\lambda} $ and thus  $ 1\geq e^{\lambda}(1-\lambda)\geq (1+\lambda)_+(1-\lambda)=(1-\lambda^2)_+$ for $ \lambda\leq 1.$ Therefore, we have \begin{align}\label{49}&1\geq\prod_j(e^{\lambda_j(B_1)}(1-\lambda_j(B_1)))\geq\prod_j(1-\lambda_j(B_1)^2)_+\geq
			1-\sum_j\lambda_j(B_1)^2.\end{align}Moreover, we have \begin{align}\label{50}&\sum_j\lambda_j(B_1)^2=|B_1|_2^2=|(\text{Id} -B)^{-1/2}(A-B)(\text{Id} -B)^{-1/2}|_2^2\\ \nonumber\leq&\|(\text{Id} -B)^{-1/2}\|^2|A-B|_2^2\|(\text{Id} -B)^{-1/2}\|^2=\|(\text{Id} -B)^{-1}\|^2|A-B|_2^2.\end{align}Therefore, the first inequality follows if we combine \eqref{48}\eqref{49}\eqref{50}. We also have\begin{align*}&|\Tr((A-B)(\text{Id}-B)^{-1})|\\=&|\Tr((A-B)+(A-B)B(\text{Id}-B)^{-1})|\\ \leq& |\Tr(A-B)|+|\Tr((A-B)B(\text{Id}-B)^{-1})|\\ \leq& |\Tr(A-B)|+|A-B|_2|B(\text{Id}-B)^{-1}|_2\\ \leq& |\Tr(A-B)|+|A-B|_2|B|_2\|(\text{Id}-B)^{-1}\|,
		\end{align*}which is the second inequality. This completes the proof.\end{proof}Thanks to Lemma \ref{lem15} and the fact that $\lim\limits_{n\to+\infty}(\ln n)^{2}e^{-(\ln n)^{1/2}}=0$, for every positive integer $k$, $x_1,\cdots,x_k\in \mathbb{R}$ and distinct $y_1,\cdots,y_k\in [0,2\pi)$, if we can prove the following bound for $n>N_1$,\begin{align}\label{51}&\Tr((A-B)(\text{Id}-B)^{-1})=0,\\ \label{52}& |A-B|_2^2=O\left(\frac{\ln n}{n^2}\right),\ \|(\text{Id}-B)^{-1}\| =O(n(\ln n)^{\frac{1}{2}}e^{-(\ln n)^{\frac{1}{2}}/2}),
	\end{align}then \eqref{47}  will be proved, and thus \eqref{42}.
	
	By \eqref{defi}, we can write\begin{align*}&A-B=\sum_{i\neq j}\chi_{I(y_i, F_n(x_i))} P_n\chi_{I(y_j, F_n(x_j))}:=\sum_{i\neq j}\chi_{i} P_n\chi_{j},
	\end{align*}here, we denote $\chi_{j}=\chi_{I(y_j, F_n(x_j))} $. For $i\neq j$, we have $\Tr(\chi_{i} P_n\chi_{j}(\text{Id}-B)^{-1})=\\ \Tr( P_n\chi_{j}(\text{Id}-B)^{-1}\chi_{i}).$ Since $I(y_j, F_n(x_j))$'s are disjoint, we have $\chi_{j}B=B_j=B\chi_{j} $ and thus $(\text{Id}-B)\chi_{j}(\text{Id}-B)^{-1}\chi_{i}=\chi_{j}(\text{Id}-B)(\text{Id}-B)^{-1}\chi_{i}=\chi_{j}\chi_{i}=0. $ Since $(\text{Id}-B)$ is invertible, we further have $\chi_{j}(\text{Id}-B)^{-1}\chi_{i}=0 $, which implies $\Tr(\chi_{i} P_n\chi_{j}(\text{Id}-B)^{-1})=0$. And thus \eqref{51} follows.
	
	By definition of $N_1$ and $z_j,$ for $x\in I(y_i, F_n(x_i)),\ y\in I(y_j, F_n(x_j)),\ i\neq j,\ n>N_1$, we have \begin{align*}&\min(|x-y|,2\pi-|x-y|)\geq \min(|y_i-y_j|,2\pi-|y_i-y_j|)-\max(F_n(x_i),F_n(x_j))\\ &\geq \min\{z_{i+1}-z_i|1\leq i\leq k\}-\min\{z_{i+1}-z_i|1\leq i\leq k\}/2\\&=\min\{z_{i+1}-z_i|1\leq i\leq k\}/2:=a_0\in(0,2\pi),\end{align*} and $$|K_n(x,y)|=\left|\dfrac{1}{2\pi}\dfrac{1-e^{in(x-y)}}{1-e^{i(x-y)}}\right| \leq\dfrac{1}{\pi}\dfrac{1}{|1-e^{i(x-y)}|}\leq \dfrac{1}{\pi}\dfrac{1}{|1-e^{ia_0}|}=O(1),$$ using this and \eqref{11} we have\begin{align*}&|A-B|_2^2=\sum_{i\neq j}\int_{I(y_i, F_n(x_i))}dx\int_{I(y_j, F_n(x_j))}|K_n(x,y)|^2dy\\ =&\sum_{i\neq j}\int_{I(y_i, F_n(x_i))}dx\int_{I(y_j, F_n(x_j))}O(1)dy=\sum_{i\neq j}F_n(x_i)F_n(x_j)O(1)\\ =&\sum_{i\neq j}O\left(\frac{\ln n}{n^2}\right)O(1)=k(k-1)O\left(\frac{\ln n}{n^2}\right)=O\left(\frac{\ln n}{n^2}\right),
	\end{align*}which is the first inequality in \eqref{52}.
	It remains to estimate $\|(\text{Id}-B)^{-1}\|, $ we need the following eigenvalue esitmate.\begin{lem}\label{lem18}Let $B$ be a finite rank symmetric operator on a Hilbert space such that $\text{Id} -B $ is positive definite, let $\lambda_1(B) $ be the largest eigenvalue of $B$, then we have $ 1-\lambda_1(B)\geq \det(\text{Id} -B)e^{\Tr B-1}.$\end{lem}\begin{proof}Let $\lambda_k(B) $ be the eigenvalues of $B$, then we have $\lambda_k(B)<1$ and\begin{align*}&\det(\text{Id} -B)e^{\Tr B-1}=\prod_k(1-\lambda_k(B))e^{\sum_k\lambda_k(B)-1}.\end{align*} Using the fact that $0<(1-\lambda)e^{\lambda}\leq 1$ for $\lambda<1$ again, we have\begin{align*}&\det(\text{Id} -B)e^{\Tr B-1}=e^{-1}\prod_k(1-\lambda_k(B))e^{\lambda_k(B)}\\=&(1-\lambda_1(B))e^{\lambda_1(B)-1}\prod_{k\neq 1}(1-\lambda_k(B))e^{\lambda_k(B)}\\ \leq &(1-\lambda_1(B))e^{\lambda_1(B)-1}\leq1-\lambda_1(B).\end{align*}This completes the proof.\end{proof}Recall the definitions of $B$ and $B_j$ in \eqref{defi}, assume $0\neq f\in L^2([0,2\pi))$ such that $ Bf=\lambda_1(B)f$ where 
	$\lambda_1(B)$ is the largest eigenvalue of $B$,  then we have\begin{align*}&\lambda_1(B)f=Bf=\sum\limits_{j=1}^kB_jf.\end{align*}
	
	For $n>N_1,\ i\neq j,$ by definition we have $ I(y_i, F_n(x_i))\cap I(y_j, F_n(x_j))=\emptyset$ and  then $\chi_{I(y_i, F_n(x_i))}B_j=0, $ thus we further have  \begin{align}\label{simd}&\lambda_1(B)\chi_{I(y_i, F_n(x_i))}f=\chi_{I(y_i, F_n(x_i))}Bf\\\nonumber &=\chi_{I(y_i, F_n(x_i))}B_if=B_if=B_i\chi_{I(y_i, F_n(x_i))}f, \end{align}
	i.e., $\chi_{I(y_i, F_n(x_i))}f$ is an eigenfunction of $B_i$ and its largest eigenvalue $\lambda_1(B_i)\geq \lambda_1(B)$.
	
	If $\chi_{I(y_i, F_n(x_i))}f\neq 0 $ for some $1\leq i\leq k,$ then by Lemma \ref{lem18} we have $1-\lambda_1(B)\geq1-\lambda_1(B_i) \geq \det(\text{Id} -B_i)e^{\Tr B_i-1}.$ Notice that $$\det(\text{Id} -B_i)=\mathbb{P}(\xi^{(n)}(I(y_i, F_n(x_i)))=0)=D_n(F_n(x_i)/2),$$ $$K_n(x,x)=n/(2\pi),$$ and
	$$\Tr B_i=\int_{I(y_i, F_n(x_i))}K_n(x,x)dx=nF_n(x_i)/(2\pi),$$ thus we have $$1-\lambda_1(B)\geq D_n(F_n(x_i)/2)e^{nF_n(x_i)/(2\pi)-1}.$$ By \eqref{11}\eqref{2} and $32>\pi^2$, there exists a constant $N_2>N_1$ such that $nF_n(x_i)>\pi(\ln n)^{\frac{1}{2}}$ and $n(4\ln n)^{\frac{1}{2}}D_n(F_n(x_i)/2)>e^{c_0-x_i}$ for  $1\leq i\leq k.$ Thus, we further have \begin{align*}&1-\lambda_1(B)\geq n^{-1}(4\ln n)^{-\frac{1}{2}}e^{c_0-x_i}e^{(\ln n)^{\frac{1}{2}}/2-1}.\end{align*} If $\chi_{I(y_i, F_n(x_i))}f= 0 $ for every $1\leq i\leq k,$ then we have $B_i f=0,$ and thus  $\lambda_1(B)f=0,$  $\lambda_1(B)=0,\ 1-\lambda_1(B)=1.$ In both cases for $n>N_2$ we always have $$1-\lambda_1(B)\geq \min(1,n^{-1}(4\ln n)^{-\frac{1}{2}}e^{c_0-\max\{x_j|1\leq j\leq k\}}e^{(\ln n)^{\frac{1}{2}}/2-1}),$$ therefore, \begin{align*}\|(\text{Id} -B)^{-1}\|&=(1-\lambda_1(B))^{-1}\leq 1+n(4\ln n)^{\frac{1}{2}}e^{\max\{x_j|1\leq j\leq k\}-c_0}e^{1-(\ln n)^{\frac{1}{2}}/2}\\ &\leq 1+O(n(\ln n)^{\frac{1}{2}}e^{-(\ln n)^{\frac{1}{2}}/2})=O(n(\ln n)^{\frac{1}{2}}e^{-(\ln n)^{\frac{1}{2}}/2}),\end{align*}which finishes the second inequality in \eqref{52}, and hence, we finish the proof of \eqref{42} in Lemma \ref{splt}.
	
	Now we can use \eqref{42} to prove the lower bound \eqref{14}. For $n>N_1$, by \eqref{19} we have \begin{align}\label{20}&\phi_{k,n}(y_1,\cdots,y_k)\geq (n/4)^k(2\ln n)^{\frac{k}{2}}\mathbb{P}(\xi^{(n)}(I_{n,k})=0)\\ \nonumber&-(n/4)^k(2\ln n)^{\frac{k}{2}}\sum_{j=1}^k\mathbb{P}(\xi^{(n)}(I_{n,k})=\xi^{(n)}(J_{n,k,j})=0).
	\end{align}Now we claim that $$(n/4)^k(2\ln n)^{\frac{k}{2}}\mathbb{P}(\xi^{(n)}(I_{n,k})=\xi^{(n)}(J_{n,k,j})=0)\to 0, \,\,\,n\to+\infty.$$ Let $x_0=\min\{x_j|1\leq j\leq k\},$ then we have $F(x_j)\geq F(x_0)$,\ $I_{n,k}\supseteq\cup_{j=1}^kI(y_j, F_n(x_0))\\=\cup_{j=1}^kI(z_j, F_n(x_0)) $. Therefore, we have  $I_{n,k}\cup J_{n,k,j}\supseteq J_{n,k,j}\cup \left(\cup_{i\neq j}I(z_i, F_n(x_0))\right),$ and the right hand side is a disjoint union for $n>N_1$. If $k=1$, then $ J_{n,k,j}=(z_1,z_1+2\pi)(\text{mod}\ 2\pi)$ and $\mathbb{P}(\xi^{(n)}(I_{n,k})=\xi^{(n)}(J_{n,k,j})=0)=\mathbb{P}(\xi^{(n)}((z_1,z_1+2\pi)(\text{mod}\ 2\pi))=0)=0. $ If $k>1,$ by \eqref{16} and \eqref{17} we have \begin{align*}0\leq&\mathbb{P}(\xi^{(n)}(I_{n,k})=\xi^{(n)}(J_{n,k,j})=0)=\mathbb{P}(\xi^{(n)}(I_{n,k}\cup J_{n,k,j})=0)\\ \leq&\mathbb{P}(\xi^{(n)}(J_{n,k,j}\cup \left(\cup_{i\neq j}I(z_i, F_n(x_0))\right))=0)\\ \leq&\mathbb{P}(\xi^{(n)}(J_{n,k,j})=0)\prod_{i\neq j}\mathbb{P}(\xi^{(n)}(I(z_i, F_n(x_0)))=0)\\=& D_n((z_{j+1}-z_j)/2)(D_n(F_n(x_0)/2))^{k-1}.
	\end{align*}Thus by \eqref{12} and \eqref{2}, we have\begin{align*}&0\leq \limsup_{n\to+\infty}(n/4)^k(2\ln n)^{\frac{k}{2}}\mathbb{P}(\xi^{(n)}(I_{n,k})=\xi^{(n)}(J_{n,k,j})=0)\\&\leq \lim_{n\to+\infty}(n/4)(2\ln n)^{\frac{1}{2}}D_n((z_{j+1}-z_j)/2)\left(\lim_{n\to+\infty}(n/4)(2\ln n)^{\frac{1}{2}}D_n(F_n(x_0)/2)\right)^{k-1}\\&=0\cdot\left(e^{c_0-x_0}/4\right)^{k-1}=0,\,\,\,\,  \forall\,\, 1\leq j\leq k,
	\end{align*}   which implies the claim. Therefore, combining the cases $ k=1$ and $k>1$, using \\ \eqref{2}\eqref{17}\eqref{42}\eqref{20}, we have \begin{align*}&\liminf_{n\to+\infty}\phi_{k,n}(y_1,\cdots,y_k)\\ \geq & \liminf_{n\to+\infty}(n/4)^k(2\ln n)^{\frac{k}{2}}\mathbb{P}(\xi^{(n)}(I_{n,k})=0)\\=&\liminf_{n\to+\infty}(n/4)^k(2\ln n)^{\frac{k}{2}}\prod_{j=1}^k\mathbb{P}(\xi^{(n)}(I(y_j, F_n(x_j)))=0)\\=&\liminf_{n\to+\infty}(n/4)^k(2\ln n)^{\frac{k}{2}}\prod_{j=1}^kD_n(F_n(x_j)/2)\\=&\prod_{j=1}^k\lim_{n\to+\infty}(n(2\ln n)^{\frac{1}{2}}D_n(F_n(x_j)/2)/4)=\prod_{j=1}^k(e^{c_0-x_j}/4),
	\end{align*} which is the lower bound \eqref{14}. Therefore, we finish the proof of Theorem \ref{thm1}.
	\section{The GUE case}\label{gu7}
	In this section, let's denote $\mathbb P^{CUE(n)}$ (or $\mathbb P^{GUE(n)}$)  as the probability taken with respect to the Haar measure of $U(n)$ (or GUE), when we drop the superscript,  the expectation $\mathbb E$ and the probability $\mathbb P$ are taken with respect to GUE.
	\subsection{Another rescaling limit}
	We first need another rescaling limit of $D_n(\alpha)$. Let's denote\begin{equation}\label{gx} G_n(x)=\frac{8x-5\ln(2\ln n)}{2n(2\ln n)^{\frac{1}{2}}}+\frac{(32\ln n)^{\frac{1}{2}}}{n}.\end{equation} Given a compact subinterval $I=[a,b]$ in $(-2,2)$, let's denote $S(I)=\inf_I\sqrt{4-x^2}$, then we have $$S(I)m_k^*=G_n(\tau_k^*),$$
	where $m_k^*$ and $\tau_k^*$ are as defined in Theorem \ref{thm2}.
	
	From the definition of $G_n(x) $ we have\begin{align}\label{30}y-x=(G_n(y)-G_n(x))(n/4)(2\ln n)^{\frac{1}{2}}, \end{align} and for every fixed $x,$ \begin{align}\label{31}\lim\limits_{n\to+\infty}\dfrac{n G_n(x)}{(32\ln n)^{\frac{1}{2}}}=1,\ \lim\limits_{n\to+\infty}n G_n(x)=+\infty,\ \lim\limits_{n\to+\infty}n^{\gamma} G_n(x)=0,\ \forall\ \gamma<1.\end{align}Now we need the following  rescaling limit which is similar to \eqref{2}. \begin{lem}\label{another}For fixed $x, z \in \mathbb R$, we have \begin{align}\label{21}\lim_{n\to+\infty}n(2\ln n)^{-\frac{1}{2}} D_n((1+z/\ln n)G_n(x)/2)=e^{c_0-x-2z}.
	\end{align}\end{lem}\begin{proof}Let $\alpha_n=(1+z/\ln n)G_n(x)/2,$ then by \eqref{31} we have $\alpha_n\to 0,\ n\alpha_n\to+\infty$ as $ n\to+\infty,$ thus $s_0/n<\alpha_n<\pi-\varepsilon $ for $n$ sufficiently large. Therefore,  \eqref{dds} holds for such $\alpha_n$, and  we still have
		\begin{align}\label{23}\lim\limits_{n\to+\infty}\left(\ln D_n(\alpha_n)-n^2\ln\cos\frac{\alpha_n}{2}+\frac{1}{4}\ln\left(n\sin\frac{\alpha_n}{2}\right)-c_0\right)=0.
		\end{align}
		By \eqref{31} we have\begin{align*}&\lim\limits_{n\to+\infty}\frac{(2\ln n)^{\frac{1}{2}}}{n\sin(\alpha_n/2)}=\lim\limits_{n\to+\infty}\frac{(2\ln n)^{\frac{1}{2}}}{n\alpha_n/2}\\&=\lim\limits_{n\to+\infty}\frac{(2\ln n)^{\frac{1}{2}}}{n(1+z/\ln n)G_n(x)/4}=\lim\limits_{n\to+\infty}\frac{(32\ln n)^{\frac{1}{2}}}{nG_n(x)}=1,\end{align*} and thus we have \begin{align}\label{24}\lim\limits_{n\to+\infty}\left(\frac{1}{8}\ln (2\ln n)-\frac{1}{4}\ln\left(n\sin\frac{\alpha_n}{2}\right)\right)=0.
		\end{align}By \eqref{31} and Taylor expansion of $\ln \cos y$ as $y\to 0$, we have \begin{align*}n^2\ln\cos\frac{\alpha_n}{2}+\frac{n^2\alpha_n^2}{8}=n^2O(\alpha_n^4)=n^2O(G_n^4(x))\to 0,\end{align*} and\begin{align*}\frac{n^2G_n^2(x)}{32}&=\frac{32\ln n}{32}+\frac{8x-5\ln(2\ln n)}{(2\ln n)^{\frac{1}{2}}}\frac{(32\ln n)^{\frac{1}{2}}}{32}+\frac{(8x-5\ln(2\ln n))^2}{32\cdot4\cdot(2\ln n)}\\&=\ln n+\frac{8x-5\ln(2\ln n)}{8}+o(1),
		\end{align*} and\begin{align*}&\frac{n^2\alpha_n^2}{8}-\frac{n^2G_n^2(x)}{32}=\frac{n^2(1+z/\ln n)^2G_n(x)^2}{32}-\frac{n^2G_n^2(x)}{32}\\=&\frac{(z/\ln n)(2+z/\ln n)n^2G_n(x)^2}{32}\\=&(z/\ln n)(2+z/\ln n)(\ln n+o(\ln n))\to2z\end{align*}as $n\to+\infty$,  which implies \begin{align}\label{25}\lim\limits_{n\to+\infty}\left(n^2\ln\cos\frac{\alpha_n}{2}+\ln n+x-\frac{5\ln(2\ln n)}{8}+2z\right)=0.
		\end{align}By \eqref{23}\eqref{24}\eqref{25} we have\begin{align*}\lim\limits_{n\to+\infty}\left(\ln D_n(\alpha_n)+\ln n+x+2z-\frac{\ln(2\ln n)}{2}-c_0\right)=0,\end{align*}and thus we have \begin{align*}\lim\limits_{n\to+\infty}\ln\left(n(2\ln n)^{-\frac{1}{2}} D_n(\alpha_n)\right)= c_0-x-2z.\end{align*}As $\alpha_n=(1+z/\ln n)G_n(x)/2$, the above limit is equivalent to \begin{align*}\lim_{n\to+\infty}n(2\ln n)^{-\frac{1}{2}} D_n((1+z/\ln n)G_n(x)/2)=e^{ c_0-x-2z},\end{align*}this completes the proof of \eqref{21}.\end{proof}
	\subsection{One integral lemma} In this subsection, we will prove one integral Lemma \ref{lem21} which will be used in the proof of Theorem \ref{thm2}. We first have the  bound, \begin{lem}\label{lem20}For every fixed $x\in\mathbb{R}$ and $A>1,$ there exists a constant $N_3>0$ depending only on $x,A$ such that for $n>N_3,\ w\in[1,A]$, we have $s_0/n<G_n(x)/2<AG_n(x)/2<\pi/2 $ and $ D_n(wG_n(x)/2)\leq e^{1-(w-1)\ln n}D_n(G_n(x)/2).$\end{lem}\begin{proof}Let $\alpha_n=G_n(x)/2$, then by \eqref{31} we have $\alpha_n\to 0,\ n\alpha_n\to+\infty$ as $ n\to+\infty,$ thus there exists a constant $N_{3,0}>0$ such that $s_0/n<\alpha_n<w\alpha_n\leq A\alpha_n<\pi/2 $ for $n>N_{3,0}$ and\begin{align*}\lim\limits_{n\to+\infty}\sup_{w\in[1,A]}\frac{1}{n\sin(w\alpha_n/2)}=
			\lim\limits_{n\to+\infty}\frac{1}{n\sin(\alpha_n/2)}=\lim\limits_{n\to+\infty}\frac{2}{n\alpha_n}=0.\end{align*}By \eqref{1} there exists a constant $N_{3,1}>N_{3,0}$ such that\begin{align*}\left|\ln D_n(w\alpha_n)-n^2\ln\cos\frac{w\alpha_n}{2}+\frac{1}{4}\ln\left(n\sin\frac{w\alpha_n}{2}\right)-c_0\right|<1/2
		\end{align*}for $n>N_{3,1},\ z\in[1,A],$ thus we have \begin{align*}&\ln (D_n(w\alpha_n)/D_n(\alpha_n))=\ln D_n(w\alpha_n)-\ln D_n(\alpha_n)\\ \leq& n^2\ln\cos\frac{w\alpha_n}{2}-\frac{1}{4}\ln\left(n\sin\frac{w\alpha_n}{2}\right) -n^2\ln\cos\frac{\alpha_n}{2}+\frac{1}{4}\ln\left(n\sin\frac{\alpha_n}{2}\right)+1 .\end{align*} Let's denote $F(y)=\ln\cos(y/2)$, 
		since $\sin\frac{w\alpha_n}{2}\geq \sin\frac{\alpha_n}{2}$, we  further  have\begin{align*}&\ln \frac{D_n(w\alpha_n)}{D_n(\alpha_n)}\leq n^2\ln\cos\frac{w\alpha_n}{2} -n^2\ln\cos\frac{\alpha_n}{2}+1=n^2(F(w\alpha_n)-F(\alpha_n))+1 .\end{align*} Since $F'(y)=-\tan(y/2)/2<-y/4\leq -\alpha_n/4$ for $ n>N_{3,1},\ y\in[\alpha_n,A\alpha_n]\subset(0,\pi),$ we have $F(w\alpha_n)-F(\alpha_n)\leq -(w\alpha_n-\alpha_n)\alpha_n/4 $ and thus \begin{align*}&\ln \frac{D_n(w\alpha_n)}{D_n(\alpha_n)}\leq -n^2(w\alpha_n-\alpha_n)\alpha_n/4+1 =-(w-1)\frac{n^2\alpha_n^2}{4}+1\end{align*}for $n>N_{3,1},\ w\in[1,A].$ By \eqref{31} we have\begin{align*}\frac{n^2\alpha_n^2}{4\ln n}=\frac{n^2G_n^2(x)}{16\ln n}\to 2\end{align*} as $n\to+\infty$, and there exists a constant $N_{3}>N_{3,1}$ such that $n^2\alpha_n^2>4\ln n $ for $n>N_3,$ which implies\begin{align*}&\ln (D_n(w\alpha_n)/D_n(\alpha_n))\leq -(w-1)\ln n+1.\end{align*}
		As $\alpha_n=G_n(x)/2$, for $n>N_3>N_{3,0}$ and $w\in[1,A]$, we have\begin{align*}&D_n(wG_n(x)/2)=D_n(w\alpha_n)=\exp(\ln (D_n(w\alpha_n)/D_n(\alpha_n)))D_n(\alpha_n)\\ \leq& e^{-(w-1)\ln n+1}D_n(\alpha_n)=e^{1-(w-1)\ln n}D_n(G_n(x)/2),\end{align*} this completes the proof.\end{proof}
	
	Using \eqref{21} and Lemma \ref{lem20}, we have the limit of the integral, \begin{lem}\label{lem21}For  $I=[a,b]\subset (-2,2)$,  let $S(I)=\inf_I\sqrt{4-y^2},$ then we have\begin{align*}\lim_{n\to+\infty}n(2\ln n)^{\frac{1}{2}}\int_ID_n(\sqrt{4-y^2}/S(I) \cdot G_n(x)/2)dy=M(I)e^{c_0-x},\end{align*}where $M(I)=(4-a^2)/|a|$ if $a+b<0,$ $M(I)=(4-b^2)/|b|$ if $a+b>0,$ and $M(I)=2(4-a^2)/|a|$ if $a+b=0.$\end{lem}\begin{proof}{\bf Case 1:} $a+b<0.$ In this case we have $a<0,\ S(I)=\sqrt{4-a^2}.$ Let $A=2/S(I)$, then we have $1\leq \sqrt{4-y^2}/S(I)\leq 2/S(I)=A $ for $y\in I.$ Let $N_3$ be determined in Lemma \ref{lem20} with $w=\sqrt{4-y^2}/S(I)\in [1, A]$, for $n>N_3,$  we have \begin{align*} D_n(\sqrt{4-y^2}/S(I) \cdot G_n(x)/2)\leq e^{1-(\sqrt{4-y^2}/S(I)-1)\ln n}D_n(G_n(x)/2). \end{align*}
		Let  $b_0=(a+b)/2$, then we have $a<b_0<\min(b,0)$ and we can write $I=I_1\cup I_2$ such that $I_1=[a,b_0]$, $I_2=[b_0,b].$ Now we have for $n$ large enough, \begin{align*}&n(2\ln n)^{\frac{1}{2}}\int_{I_2}D_n(\sqrt{4-y^2}/S(I) \cdot G_n(x)/2)dy\\ \leq&n(2\ln n)^{\frac{1}{2}} \int_{I_2}e^{1-(\sqrt{4-y^2}/S(I)-1)\ln n}D_n(G_n(x)/2)dy\\ \leq&n(2\ln n)^{\frac{1}{2}} \int_{I_2}e^{1-(S(I_2)/S(I)-1)\ln n}D_n(G_n(x)/2)dy\\=&n(2\ln n)^{\frac{1}{2}}(b-b_0)e^{1-(S(I_2)/S(I)-1)\ln n}D_n(G_n(x)/2),\end{align*} where  $S(I_2)=\min(\sqrt{4-b_0^2},\sqrt{4-b^2})>S(I)>0$. By \eqref{21},  we have \begin{align*}&\lim_{n\to+\infty}n(2\ln n)^{\frac{1}{2}}e^{1-(S(I_2)/S(I)-1)\ln n}D_n(G_n(x)/2)\\=&\lim_{n\to+\infty}(2\ln n)e^{1-(S(I_2)/S(I)-1)\ln n}\lim_{n\to+\infty}n(2\ln n)^{-\frac{1}{2}}D_n(G_n(x)/2)=0,\end{align*}which implies\begin{align}\label{26}\lim_{n\to+\infty}n(2\ln n)^{\frac{1}{2}}\int_{I_2}D_n(\sqrt{4-y^2}/S(I) \cdot G_n(x)/2)dy=0.\end{align}
		As to the integration in $I_1$, we change variable $y=-\sqrt{4-z^2}$ to obtain\begin{align*}&\int_{I_1}D_n(\sqrt{4-y^2}/S(I) \cdot G_n(x)/2)dy= \int_{a_1}^{b_1}D_n(z/S(I) \cdot G_n(x)/2)\frac{z}{\sqrt{4-z^2}}dz\\ =& S(I)(\ln n)^{-1}\int_{0}^{(b_1/a_1-1)\ln n}D_n((1+z/\ln n) G_n(x)/2)\frac{(1+z/\ln n)a_1}{\sqrt{4-(1+z/\ln n)^2a_1^2}}dz,\end{align*}here $a_1=\sqrt{4-a^2}=S(I),$ $b_1=\sqrt{4-b_0^2}>a_1,$ thus we have \begin{align*}&n(2\ln n)^{\frac{1}{2}}\int_{I_1}D_n(\sqrt{4-y^2}/S(I) \cdot G_n(x)/2)dy=n(2\ln n)^{-\frac{1}{2}}D_n(G_n(x)/2)\times\\ &2S(I)\int_{0}^{(b_1/a_1-1)\ln n}\frac{D_n((1+z/\ln n) G_n(x)/2)}{D_n(G_n(x)/2)}\frac{(1+z/\ln n)a_1}{\sqrt{4-(1+z/\ln n)^2a_1^2}}dz.\end{align*}Since $b_1/a_1=b_1/S(I)\leq 2/S(I)=A,$ by Lemma \ref{lem20} the integrand above has the uniform bound\begin{align*}&\sup_{z\in [0,(b_1/a_1-1)\ln n]}e^z\frac{D_n((1+z/\ln n) G_n(x)/2)}{D_n(G_n(x)/2)}\frac{(1+z/\ln n)a_1}{\sqrt{4-(1+z/\ln n)^2a_1^2}}\\=&\sup_{w\in[1,b_1/a_1]}e^{(w-1)\ln n}\frac{D_n(w G_n(x)/2)}{D_n(G_n(x)/2)}\frac{wa_1}{\sqrt{4-w^2a_1^2}}\\ \leq&\sup_{w\in[1,b_1/a_1]}e^{(w-1)\ln n}e^{1-(w-1)\ln n}\frac{b_1}{\sqrt{4-b_1^2}}=\frac{eb_1}{\sqrt{4-b_1^2}}\end{align*} for $n$ large enough. By \eqref{21}  (with $z=0$) we have \begin{align*}&\lim_{n\to+\infty}\frac{D_n((1+z/\ln n) G_n(x)/2)}{D_n(G_n(x)/2)}\frac{(1+z/\ln n)a_1}{\sqrt{4-(1+z/\ln n)^2a_1^2}}\\=&\lim_{n\to+\infty}\frac{D_n((1+z/\ln n) G_n(x)/2)}{D_n(G_n(x)/2)}\frac{a_1}{\sqrt{4-a_1^2}}=e^{-2z}\frac{a_1}{\sqrt{4-a_1^2}}.\end{align*}Therefore, we can apply the dominated convergence theorem to get \begin{align*}&\lim_{n\to+\infty}2S(I)\int_{0}^{(b_1/a_1-1)\ln n}\frac{D_n((1+z/\ln n) G_n(x)/2)}{D_n(G_n(x)/2)}\frac{(1+z/\ln n)a_1}{\sqrt{4-(1+z/\ln n)^2a_1^2}}dz\\=&2S(I)\int_{0}^{+\infty}e^{-2z}\frac{a_1}{\sqrt{4-a_1^2}}dz=
			\frac{S(I)a_1}{\sqrt{4-a_1^2}}=\frac{S(I)\sqrt{4-a^2}}{|a|},\end{align*} and by \eqref{21}  with $z=0$ again,  we have \begin{align*}&\lim_{n\to+\infty}n(2\ln n)^{-\frac{1}{2}}D_n(G_n(x)/2)=e^{c_0-x},\end{align*}which implies\begin{align}\label{27}\lim_{n\to+\infty}n(2\ln n)^{\frac{1}{2}}\int_{I_1}D_n(\sqrt{4-y^2}/S(I) \cdot G_n(x)/2)dy=e^{c_0-x}S(I)\frac{\sqrt{4-a^2}}{|a|},\end{align}which finishes the proof by the fact that $S(I)\frac{\sqrt{4-a^2}}{|a|}=({4-a^2})/{|a|}=M(I). $
		
		{\bf Case 2:} $a+b>0.$  By symmetry, we can consider $-I=[-b,-a]$ 
		and the result follows {\bf Case 1}.
		
		{\bf Case 3:} $a+b=0.$ We can write $I=I_1\cup I_2$ such that $I_1=[a,0]$, $I_2=[0,b]$, then we have $S(I)=S(I_1)=S(I_2),\ M(I)=M(I_1)+M(I_2)$, and by the results of {\bf Case 1}, {\bf Case 2} we have\begin{align*}&n(2\ln n)^{\frac{1}{2}}\int_ID_n(\sqrt{4-y^2}/S(I) \cdot G_n(x)/2)dy\\=&\sum_{j=1}^2n(2\ln n)^{\frac{1}{2}}\int_{I_j}D_n(\sqrt{4-y^2}/S(I_j) \cdot G_n(x)/2)dy\\ \to& M(I_1)e^{c_0-x}+M(I_2)e^{c_0-x}=M(I)e^{c_0-x},\,\,\,n\to+\infty, \end{align*} this completes the proof.\end{proof}
	
	\subsection{The strategy to prove Theorem \ref{thm2}}
	The strategy to prove Theorem \ref{thm2} is similar to that of Theorem \ref{thm1}, but we will still give all the detailed definitions and computations. Now we consider the point process of eigenvalues of GUE,  $$ \xi^{(n)}=\sum_{i=1}^n\delta_{\lambda_i}.$$ 
	By definition of $M_0(I)$ in Theorem \ref{thm2} and $M(I)$ in Lemma \ref{lem21}, we have $M_0(I)=\ln(M(I)S(I)/4) .$ Take $c_2=c_0+M_0(I),\ f(x)=e^{c_2-x}=M(I)S(I)e^{c_0-x}/4,$ then we have $-f'(x)=f''(x)=e^{c_2-x}. $ By Lemma \ref{lem1}, for every positive integer $k$ and $x_1,\cdots,x_k\in \mathbb{R}$, for $\tau_{j}^*$ defined in Theorem \ref{thm2}, if we can prove the following convergence\begin{align}\label{61}&\lim_{n\to+\infty}\mathbb{E}\sum_{i_1,\cdots,i_{k}\ \text{all distinct}}\prod_{j=1}^k( {\tau}^*_{i_j}-x_j)_+
		=(M(I)S(I))^k\prod_{j=1}^k\left(e^{c_0-x_j}/4\right),
	\end{align}then Theorem \ref{thm2} will be proved.
	
	For $\lambda_1 < \cdots < \lambda_n  $, denote $J_k(a):=\{x\in \mathbb{R}|[x,x+a]\subset(\lambda_k,\lambda_{k+1})\}$ for $a>0,\ 1\leq k< n,$ then we have $J_k(a)=(\lambda_k,\lambda_{k+1}-a)$ for $\lambda_{k+1}-\lambda_k>a $ and $J_k(a)=\emptyset$ for $\lambda_{k+1}-\lambda_k\leq a, $ thus $J_k(a)$ is an interval of size $(\lambda_{k+1}-\lambda_k-a)_+,$ and $J_k(a)\subset(\lambda_k,\lambda_{k+1}) $ and $J_k(a)\cap J_l(a)=\emptyset $ for $k\neq l.$ Now let $ \Lambda(I)=\{i|\lambda_{i},\lambda_{i+1}\in I\}$,\begin{align*}&\Sigma_k(a_1,\cdots,a_k):=\bigcup_{i_1,\cdots,i_{k}\in\Lambda(I)\ \text{all distinct}}\prod_{j=1}^kJ_{i_j}(a_j)\subset (a,b)^k,
	\end{align*}then the right hand side is a disjoint union and\begin{align*}|\Sigma_k(a_1,\cdots,a_k)|&=\sum_{i_1,\cdots,i_{k}\in\Lambda(I)\ \text{all distinct}}\prod_{j=1}^k(\lambda_{i_j+1}-\lambda_{i_j}-a_j)_+\\ &=\sum_{i_1,\cdots,i_{k}\in\Lambda(I)\ \text{all distinct}}\prod_{j=1}^k( {m}^*_{i_j}-a_j)_+.
	\end{align*} Let $A=2/S(I)>1,$ thanks to Lemma \ref{lem20}, for every fixed $x_1,\cdots,x_k\in \mathbb{R}$ there exists $N_3>0$ such that $0<2s_0/n<G_n(x_j)<AG_n(x_j)<\pi$ for $n>N_3,\ 1\leq j\leq k.$ Now we always assume $n>N_3.$ By \eqref{30} and the fact that $S(I){m}_k^*=G_n( {\tau}^*_k)$, we have $ {\tau}^*_k-x=(G_n( {\tau}^*_k)-G_n(x))(n/4)(2\ln n)^{\frac{1}{2}}=(S(I) {m}^*_k-G_n(x))(n/4)(2\ln n)^{\frac{1}{2}}$, and \begin{align*}&\sum_{i_1,\cdots,i_{k}\in\Lambda(I)\ \text{all distinct}}\prod_{j=1}^k( {\tau}^*_{i_j}-x_j)_+
		\\=&(nS(I)/4)^k(2\ln n)^{\frac{k}{2}}\sum_{i_1,\cdots,i_{k}\in\Lambda(I)\ \text{all distinct}}\prod_{j=1}^k( {m}^*_{i_j}-G_n(x_j)/S(I))_+\\=&(nS(I)/4)^k(2\ln n)^{\frac{k}{2}}|\Sigma_k(G_n(x_1)/S(I),\cdots,G_n(x_k)/S(I))|.
	\end{align*}For fixed $x_1,\cdots,x_k\in \mathbb{R}$ and $y_1,\cdots,y_k\in I$, let \begin{align*} &\phi_{k,n}(y_1,\cdots,y_k)=n^k(2\ln n)^{\frac{k}{2}}\times \\&  \mathbb{P}((y_1,\cdots,y_k)\in\Sigma_k(G_n(x_1)/S(I),\cdots,G_n(x_k)/S(I))),\end{align*}  then\begin{align*}&\mathbb{E}\sum_{i_1,\cdots,i_{k}\ \text{all distinct}}\prod_{j=1}^k({\tau}^*_{i_j}-x_j)_+
		\\=&\mathbb{E}(nS(I)/4)^k(2\ln n)^{\frac{k}{2}}|\Sigma_k(G_n(x_1)/S(I),\cdots,G_n(x_k)/S(I))|\\=&(S(I)/4)^k\int_{I^k}\phi_{k,n}(y_1,\cdots,y_k)dy_1\cdots dy_k.
	\end{align*}Now we prove the following upper bound and lower bound separately
	\begin{align}\label{33}&\limsup_{n\to+\infty}\int_{I^k}\phi_{k,n}(y_1,\cdots,y_k)dy_1\cdots dy_k\leq (M(I))^k\prod_{j=1}^k\left(e^{c_0-x_j}\right),
	\end{align}  \begin{align}\label{34}&\liminf_{n\to+\infty}\int_{I^k}\phi_{k,n}(y_1,\cdots,y_k)dy_1\cdots dy_k\geq (M(I))^k\prod_{j=1}^k\left(e^{c_0-x_j}\right),
	\end{align}in fact \eqref{33} and \eqref{34} imply \eqref{61}, and thus Theorem \ref{thm2} follows.
	\subsection{The proof of Theorem \ref{thm2}}
	We first need  the following  equivalent condition  for a point in $\Sigma_k(a_1,\cdots,a_k)$, the proof is similar  to that of  Lemma \ref{lem2} and we omit it here.  \begin{lem}\label{lem22}For $(y_1,\cdots,y_k)\in(a,b)^k, $ the condition $(y_1,\cdots,y_k)\in\Sigma_k(a_1,\cdots,a_k) $ is equivalent to the following conditions: (i) $[y_l,y_l+a_l]\cap [y_j,y_j+a_j]=\emptyset $ for $1\leq l<j\leq k,$ and (ii) $ \lambda_l\not\in [y_l,y_l+a_l],$ for $1\leq j\leq k,\ 1\leq l\leq n$, and (iii) $\{ \lambda_1,\cdots ,\lambda_n \}\cap[y_p,y_q]\neq\emptyset,$ for every $p,q\in\{0,\cdots,k+1\},$ such that $y_p<y_q,$ here we denote $y_0=a,y_{k+1}=b.$\end{lem}

	\subsubsection{Upper bound}
	Now for fixed $x_1,\cdots,x_k\in \mathbb{R}$, as $n$ large enough, let \begin{align}\label{andd}A_n:=&\{(y_1,\cdots,y_k)\in(a,b)^k|[y_i,y_i+G_n(x_i)/S(I)]\\ \nonumber & \cap [y_j,y_j+G_n(x_j)/S(I)]=\emptyset,\forall\ 1\leq i<j\leq k\},\end{align} then for $ (y_1,\cdots,y_k)\in (a,b)^k\setminus A_n,$ by Lemma \ref{lem22} we have $\phi_{k,n}(y_1,\cdots,y_k)=0. $ If $ (y_1,\cdots,y_k)\in A_n,$ then all $y_k$'s are distinct, let $y_0=a,\ y_{k+1}=b,$ and\begin{align}\label{58}I_{n,k}=\cup_{j=1}^k[y_j,y_j+G_n(x_j)/S(I)],\ J_{n,k,j}=[z_j,z_{j+1}],\ 0\leq j\leq k,\end{align} here $z_j\ (0\leq j\leq k+1)$ is the increasing rearrangement of $y_j\ (0\leq j\leq k+1),$ then $I_{n,k}$ is a disjoint union and by Lemma \ref{lem22} we have\begin{align}\label{59}&\phi_{k,n}(y_1,\cdots,y_k)=n^k(2\ln n)^{\frac{k}{2}}\times\\&\nonumber \mathbb{P}(\xi^{(n)}(I_{n,k})=0,\ \xi^{(n)}(J_{n,k,j})>0,\ \forall\ 0\leq j\leq k).
	\end{align}By Lemma \ref{lem3} and \eqref{16} we have, \begin{align*}&\phi_{k,n}(y_1,\cdots,y_k)\leq n^k(2\ln n)^{\frac{k}{2}}\mathbb{P}(\xi^{(n)}(I_{n,k})=0)\\ &\leq n^k(2\ln n)^{\frac{k}{2}}\prod_{j=1}^k\mathbb{P}(\xi^{(n)}([y_j,y_j+G_n(x_j)/S(I)])=0).
	\end{align*}and this inequality is clearly true for $(y_1,\cdots,y_k)\not\in A_n .$ Therefore, we have \begin{align*}&\int_{I^k}\phi_{k,n}(y_1,\cdots,y_k)dy_1\cdots dy_k\\ \leq&\int_{I^k}n^k(2\ln n)^{\frac{k}{2}}\prod_{j=1}^k\mathbb{P}(\xi^{(n)}([y_j,y_j+G_n(x_j)/S(I)])=0)dy_1\cdots dy_k\\ =&\prod_{j=1}^k\left[n(2\ln n)^{\frac{1}{2}}\int_{I}\mathbb{P}(\xi^{(n)}([y_j,y_j+G_n(x_j)/S(I)])=0)dy_j\right].
	\end{align*}Thus, \eqref{33} follows if we can prove the following inequality
	\begin{align}\label{22}&\limsup_{n\to+\infty}n(2\ln n)^{\frac{1}{2}}\int_{I}\mathbb{P}(\xi^{(n)}([y,y+G_n(x)/S(I)])=0)dy\leq M(I)e^{c_0-x},
	\end{align}and by Lemma \ref{lem21}, we only need to prove \begin{align}\label{32}&\limsup_{n\to+\infty}n(2\ln n)^{\frac{1}{2}}\sup_{y\in I}\Big(\mathbb{P}(\xi^{(n)}([y,y+G_n(x)/S(I)])=0)\\ \nonumber&-D_n(\sqrt{4-y^2}/S(I) \cdot G_n(x)/2)\Big)\leq 0.
	\end{align}
	Let $\{h_n\}$ be
	the Hermite polynomials, which are the successive monic orthogonal
	polynomials with respect to the Gaussian weight $e^{-x^2/2}dx.$ Following \cite{AGZ}, we introduce the functions $$\psi_k(x)=\frac{e^{-x^2/4}}{\sqrt{\sqrt{2\pi}k!}}h_k(x).$$ Then the set of points $\{\lambda_1,\cdots,\lambda_n\}$ with respect to the joint density \eqref{6} is a determinantal point
	process with the kernel given by \cite{AGZ}\begin{align}\label{29}&K^{GUE(n)}(x,y)=\sqrt{n}\frac{\psi_{n}(x\sqrt{n})\psi_{n-1}(y\sqrt{n})-
			\psi_{n-1}(x\sqrt{n})\psi_{n}(y\sqrt{n})}{x-y}.\end{align}The probability that  $\xi^{(n)}$ has no point in a measurable subset $J$ is $$\mathbb{P}(\xi^{(n)}(J)=0):=\mathbb{P}^{GUE(n)}(\lambda_i\not\in J,1\leq i\leq n)=\det (\text{Id} -\chi_J P_{GUE(n)}\chi_J), $$ where  $P_{GUE(n)}$ is the orthogonal projection from $L^2(\mathbb{R})$ to   $W_n:=\text{span}\{x^{k}e^{-nx^2/4}\\|0\leq k<n,k\in\mathbb{Z}\}$ with kernel $K^{GUE(n)}(x,y)$. 
	
	We will need the following inequality regarding the difference of the gap probabilities between CUE and GUE, 
	\begin{lem}\label{lem23}Let $ \varepsilon_0\in(0,1),\ C_0>c_*>0,\ \rho_{sc}(x)=\sqrt{(4-x^2)_+}/(2\pi).$ Then uniformly for $x\in(-2+\varepsilon_0,2-\varepsilon_0),\ c_*(\ln n)^{\frac{1}{2}}/n<\delta_n<\min(C_0(\ln n)^{\frac{1}{2}}/n,1/2),$\begin{align*}&\mathbb{P}^{GUE(n)}(\lambda_i\not\in [x,x+\delta_n/\rho_{sc}(x)],1\leq i\leq n)\\&-\mathbb{P}^{CUE(n)}(\theta_i\not\in [0,2\pi\delta_n],1\leq i\leq n)\\&\leq O((n\ln n)^{-1}).
	\end{align*}\end{lem}\begin{proof}Let $A,\ B$ be integral operators with
		respective kernels\begin{align*}&A(u,v)=-\frac{1}{n\rho_{sc}(x)}K_{(0,n\delta_n)}^{GUE(n)}\left(x+\frac{u}{n\rho_{sc}(x)},
			x+\frac{v}{n\rho_{sc}(x)}\right)
		\end{align*}and\begin{align*}&B(u,v)=-\frac{2\pi}{n}K_{(0,n\delta_n)}^{CUE(n)}\left(\frac{2\pi}{n}u,\frac{2\pi}{n}v\right).
		\end{align*}From the proof of Lemma 3.5 in \cite{BB}, we know that\begin{align}\label{35}&|A-B|_2=O((\ln n)^{3/2}/n),\ |A|_2^2=O((\ln n)^{2/3}),\ |B|_2^2=O((\ln n)^{2/3}),\\& \label{36}\Tr A=-n\delta_n+O((\ln n)^{3/2}/n),\ \Tr B=-n\delta_n+O((\ln n)^{3/2}/n).
		\end{align} We also have\begin{align}\label{37}&\det(\text{Id}+A)=\mathbb{P}^{GUE(n)}(\lambda_i\not\in [x,x+\delta_n/\rho_{sc}(x)],1\leq i\leq n)
		\end{align}and\begin{align}\label{38}&\det(\text{Id}+B)=\mathbb{P}^{CUE(n)}(\theta_i\not\in [0,2\pi\delta_n],1\leq i\leq n)=D_n(\pi\delta_n).
		\end{align}Since $D_n(\alpha)$ is a continuous function for $ \alpha\in[0,\pi]$, $D_n(0)=1$ and $D_n(\pi)=0,$ for $n\geq 2$ there exists $\alpha_n\in(0,\pi)$ such that $D_n(\alpha_n)=(n\ln n)^{-1}.$ Now we discuss the case $\pi\delta_n\leq \alpha_n$ and the case $\pi\delta_n\geq \alpha_n$ separately.
		
		If $\pi\delta_n\leq \alpha_n$,  recall the general comparison inequalities in Lemma \ref{lem15}, we have \begin{align}\label{39}\exp(\Tr(B-A)(\text{Id}+B)^{-1})\det(\text{Id} +A)/\det(\text{Id}+B)\leq1, \end{align} and \begin{align}\label{40} |\Tr((B-A)(\text{Id}+B)^{-1})|\leq|\Tr(A-B)|+|A-B|_2|B|_2\|(\text{Id}+B)^{-1}\|. \end{align} By Lemma \ref{lem18} we have\begin{align}\label{53}&\|(\text{Id}+B)^{-1}\|=(1-\lambda_1(-B))^{-1}\leq e^{1+\Tr B}(\det(\text{Id}+B))^{-1}.
		\end{align}Since $D_n(\alpha)$ is decreasing and $\pi\delta_n\leq \alpha_n$, by \eqref{38} we have\begin{align}\label{54}&\det(\text{Id}+B)=D_n(\pi\delta_n)\geq D_n(\alpha_n)=(n\ln n)^{-1}.
		\end{align}By \eqref{35}\eqref{36}\eqref{40}\eqref{53} and the fact that $c_*(\ln n)^{\frac{1}{2}}/n<\delta_n $, we have\begin{align}\label{55}&|\Tr((B-A)(\text{Id}+B)^{-1})|\\ \nonumber\leq& O((\ln n)^{3/2}/n)+O((\ln n)^{3/2+1/3}/n)e^{1+\Tr B}(\det(\text{Id}+B))^{-1}\end{align} and we also have \begin{align}\label{56}&\det(\text{Id}+B)\leq e^{\Tr B}=e^{-n\delta_n+O((\ln n)^{3/2}/n)}=e^{O(1)-c_{*}(\ln n)^{1/2}}=O((\ln n)^{-3}).\end{align}By \eqref{54}\eqref{55}\eqref{56}, we have\begin{align}\label{ttd}&|\Tr((B-A)(\text{Id}+B)^{-1})|\\ \nonumber \leq& O((\ln n)^{3/2}/n)+O((\ln n)^{3/2+1/3-3}/n)(\det(\text{Id}+B))^{-1}\\ \nonumber \leq& O((\ln n)^{2}/n)+O((\ln n)^{-7/6}/n)(n \ln n)=O(1),
		\end{align}and thus we have \begin{align*}|\exp(-\Tr(B-A)(\text{Id}+B)^{-1})-1|=O(|\Tr((B-A)(\text{Id}+B)^{-1})|), \end{align*}and we further have (using \eqref{39}\eqref{56}\eqref{ttd})\begin{align*}&\det(\text{Id} +A)-\det(\text{Id}+B)\\ \leq&\exp(-\Tr(B-A)(\text{Id}+B)^{-1})\det(\text{Id}+B)-\det(\text{Id}+B)\\ \leq&O(|\Tr(B-A)(\text{Id}+B)^{-1}|)\det(\text{Id}+B)\\ \leq& O((\ln n)^{3/2}/n)\det(\text{Id}+B)+O((\ln n)^{3/2+1/3-3}/n)\\ \leq& O((\ln n)^{2}/n)O((\ln n)^{-3})+O((\ln n)^{-1}/n)=O((n\ln n)^{-1}).\end{align*}Now the result follows from the identities \eqref{37} and \eqref{38}.
		
		If $\pi\delta_n\geq \alpha_n$, then we have (taking $\delta_n'=\alpha_n/\pi\leq \delta_n$)\begin{align*}&\mathbb{P}^{GUE(n)}(\lambda_i\not\in [x,x+\delta_n/\rho_{sc}(x)],1\leq i\leq n)\\ \leq& \mathbb{P}^{GUE(n)}(\lambda_i\not\in [x,x+\delta_n'/\rho_{sc}(x)],1\leq i\leq n)\\ \leq&\mathbb{P}^{CUE(n)}(\theta_i\not\in [0,2\pi\delta_n'],1\leq i\leq n)+O((n\ln n)^{-1})= O((n\ln n)^{-1}),
		\end{align*}and the result is also true, here we used the fact that \begin{align*}&\mathbb{P}^{CUE(n)}(\theta_i\not\in [0,2\pi\delta_n'],1\leq i\leq n)=D_n(\pi\delta_n')=D_n(\alpha_n)=(n\ln n)^{-1}.
		\end{align*}This completes the proof.\end{proof}Now we prove \eqref{32}. For $y\in I,\ x\in \mathbb{R},$ take $ \delta_n=[\sqrt{4-y^2}/S(I)] \cdot [G_n(x)/(2\pi)]$, then we have $ \delta_n/\rho_{sc}(y)=2\pi\delta_n/\sqrt{4-y^2}=G_n (x)/S(I).$ By \eqref{31}, there exists a constant $N_4>0$ depending only on $x$ such that $4(\ln n)^{\frac{1}{2}}/n<G_n (x)<8(\ln n)^{\frac{1}{2}}/n <\pi S(I)/2 $ for $n>N_4.$ Then we have 
	$(2/\pi)(\ln n)^{\frac{1}{2}}/n<G_n(x)/(2\pi)\leq [\sqrt{4-y^2}/S(I)] \cdot [G_n(x)/(2\pi)]=\delta_n\leq [2/S(I)] \cdot [G_n(x)/(2\pi)]< (\pi S(I))^{-1} \cdot 8(\ln n)^{\frac{1}{2}}/n<1/2 $ for $y\in I,\ n>N_4,$
	thus by Lemma \ref{lem23} we deduce that \begin{align*}&\mathbb{P}(\xi^{(n)}([y,y+G_n(x)/S(I)])=0)-D_n(\sqrt{4-y^2}/S(I) \cdot G_n(x)/2)\\=&\mathbb{P}(\xi^{(n)}([y,y+\delta_n/\rho_{sc}(y)])=0)-D_n(\pi\delta_n)\\=&\mathbb{P}^{GUE(n)}(\lambda_i\not\in [y,y+\delta_n/\rho_{sc}(y)],1\leq i\leq n)\\&-\mathbb{P}^{CUE(n)}(\theta_i\not\in [0,2\pi\delta_n],1\leq i\leq n)\leq O((n\ln n)^{-1}),
	\end{align*}and the estimate is uniform for  $y\in I,\ n>N_4.$ Thus we have \begin{align*}&n(2\ln n)^{\frac{1}{2}}\sup_{y\in I}\Big(\mathbb{P}(\xi^{(n)}([y,y+G_n(x)/S(I)])=0)-D_n(\sqrt{4-y^2}/S(I) \cdot G_n(x)/2)\Big)\\ &\leq n(2\ln n)^{\frac{1}{2}} O((n\ln n)^{-1})=O((\ln n)^{-1/2})\to 0,\,\,\,n\to+\infty,
	\end{align*}  and thus \eqref{32} is true, so is  \eqref{22} and hence the upper bound \eqref{33}.
	\subsubsection{Lower bound}
	For the lower bound \eqref{34}, we discuss the 3 cases separately.
	
	{\bf Case 1:} $a+b<0.$ Let $b_0=(a+b)/2<0,\ I_1=(a,b_0)\subset I,\ a_{*}=\sqrt{4-a^2}=S(I),$ $b_*=\sqrt{4-b_0^2}>a_*.$  We change variables $y_j=-\sqrt{4-v_j^2},\ 0<v_j=(1+u_j/\ln n)a_*$ to obtain\begin{align*}&\int_{I^k}\phi_{k,n}(y_1,\cdots,y_k)dy_1\cdots dy_k\geq \int_{I_1^k}\phi_{k,n}(y_1,\cdots,y_k)dy_1\cdots dy_k\\ =& \int_{(a_*,b_*)^k}\phi_{k,n}\left(-\sqrt{4-v_1^2},\cdots,-\sqrt{4-v_k^2}\right)\prod_{j=1}^k\frac{v_j}{\sqrt{4-v_1^2}}dv_1\cdots dv_k\\ =& a_*^k(\ln n)^{-k}\int_{(0,(b_*/a_*-1)\ln n)^k}\phi_{k,n}\bigg(-\sqrt{4-(1+u_1/\ln n)^2a_*^2},\cdots,\\&-\sqrt{4-(1+u_k/\ln n)^2a_*^2}\bigg)\prod_{j=1}^k\frac{(1+u_j/\ln n)a_*}{\sqrt{4-(1+u_j/\ln n)^2a_*^2}}du_1\cdots du_k.\end{align*}Denote $l_n=(b_*/a_*-1)\ln n$ and \begin{align}\label{57}&\gamma_n(u)=-\sqrt{4-(1+u/\ln n)^2S(I)^2},\ \ \beta_n(u)=\frac{(1+u/\ln n)S(I)}{\sqrt{4-(1+u/\ln n)^2S(I)^2}},\end{align} then $ \gamma_n$ maps $(0,l_n) $ to $I_1\subset (a,b)$ and \begin{align}\label{62}&\int_{I^k}\phi_{k,n}(y_1,\cdots,y_k)dy_1\cdots dy_k \\ \nonumber\geq& S(I)^k(\ln n)^{-k}\int_{(0,l_n)^k}\phi_{k,n}\big(\gamma_n(u_1),\cdots,\gamma_n(u_k)\big)\prod_{j=1}^k\beta_n(u_j)du_1\cdots du_k.\end{align}
	
	{\bf Case 2:} $a+b>0.$ Let $b_0=(a+b)/2>0,\ I_1=(b_0,b)\subset I,\ a_*=\sqrt{4-b^2}=S(I),$ $b_*=\sqrt{4-b_0^2}>a_*,$ $l_n=(b_*/a_*-1)\ln n$ and \begin{align}\label{63}&\gamma_n(u)=\sqrt{4-(1-u/\ln n)^2S(I)^2},\ \ \beta_n(u)=\frac{(1-u/\ln n)S(I)}{\sqrt{4-(1-u/\ln n)^2S(I)^2}}.\end{align}Similar to {\bf Case 1} we have $ \gamma_n:(-l_n,0)\to I_1\subset (a,b)$ and\begin{align}\label{64}&\int_{I^k}\phi_{k,n}(y_1,\cdots,y_k)dy_1\cdots dy_k\geq \int_{I_1^k}\phi_{k,n}(y_1,\cdots,y_k)dy_1\cdots dy_k \\ \nonumber=& S(I)^k(\ln n)^{-k}\int_{(-l_n,0)^k}\phi_{k,n}\big(\gamma_n(u_1),\cdots,\gamma_n(u_k)\big)\prod_{j=1}^k\beta_n(u_j)du_1\cdots du_k.\end{align}
	
	{\bf Case 3:} $a+b=0.$ Let $a_0=a/2<0,\ b_0=b/2=-a_0>0,\ I_1=(a,a_0)\cup(b_0,b)\subset I,\ a_*=\sqrt{4-a^2}=\sqrt{4-b^2}=S(I),$ $b_*=\sqrt{4-a_0^2}=\sqrt{4-b_0^2}>a_*,$ $l_n=(b_*/a_*-1)\ln n$ and functions $\gamma_n(u),\beta_n(u) $ be defined as \eqref{57} for $u>0$ and as \eqref{63} for $u<0$. Similar to {\bf Case 1} we have $ \gamma_n:(-l_n,l_n)\setminus\{0\}\to I_1\subset (a,b)$ and\begin{align}\label{65}&\int_{I^k}\phi_{k,n}(y_1,\cdots,y_k)dy_1\cdots dy_k\geq \int_{I_1^k}\phi_{k,n}(y_1,\cdots,y_k)dy_1\cdots dy_k \\ \nonumber=& S(I)^k(\ln n)^{-k}\int_{(-l_n,l_n)^k}\phi_{k,n}\big(\gamma_n(u_1),\cdots,\gamma_n(u_k)\big)\prod_{j=1}^k\beta_n(u_j)du_1\cdots du_k.\end{align}
	Now the lower bound \eqref{34} is the consequence of the following  \begin{lem}\label{lem24}For fixed $I=[a,b]\subset(-2,2),\ k\in\mathbb{Z},\ k>0,$ $x_1,\cdots,x_k\in\mathbb{R},$ let $\gamma_n(u)$ be defined as \eqref{57} for $u>0$ and as \eqref{63} for $u<0$. Assume that (i) $a+b<0,$ $u_1,\cdots,u_k\in(0,+\infty)$ all distinct, or (ii) $a+b>0,$ $u_1,\cdots,u_k\in(-\infty,0)$ all distinct, or (iii) $a+b=0,$ $u_1,\cdots,u_k\in\mathbb{R}\setminus\{0\}$ and $|u_j|$'s are all distinct, then we have\begin{align*}&\liminf_{n\to+\infty}(\ln n)^{-k}\phi_{k,n}\big(\gamma_n(u_1),\cdots,\gamma_n(u_k)\big)\geq 2^ke^{\sum_{j=1}^k(c_0-x_j-2|u_j|)}.
	\end{align*}\end{lem}
	
	Lemma \ref{lem24} will imply the lower bound \eqref{34} as follows.
	
	For the case $a+b<0,$ denote $I_0=(0,+\infty),$ then we have $\int_{I_0}2e^{-2|u|}du=1,\ S(I)=\sqrt{4-a^2}$ and $S(I)^2/{\sqrt{4-S(I)^2}} =(4-a^2)/|a| =M(I)
	.$
	
	Since $l_n\to+\infty,\ \beta_n(u_j)\to{S(I)}/{\sqrt{4-S(I)^2}}$ as $n\to+\infty,$ by \eqref{62}, Lemma \ref{lem24} and Fatou's Lemma, we have\begin{align*}&\liminf_{n\to+\infty}\int_{I^k}\phi_{k,n}(y_1,\cdots,y_k)dy_1\cdots dy_k\\ \geq& S(I)^k\int_{I_0^k}\liminf_{n\to+\infty}\left[(\ln n)^{-k}\phi_{k,n}\big(\gamma_n(u_1),\cdots,\gamma_n(u_k)\big)\prod_{j=1}^k\beta_n(u_j)\right]du_1\cdots du_k\\ \geq& S(I)^k\int_{I_0^k}2^ke^{\sum_{j=1}^k(c_0-x_j-2|u_j|)}\left({S(I)}/{\sqrt{4-S(I)^2}}\right)^kdu_1\cdots du_k\\ =& \left({S(I)}^2/{\sqrt{4-S(I)^2}}\right)^k\prod_{j=1}^k\left(e^{c_0-x_j}\right)\int_{I_0^k}\prod_{j=1}^k\left(2e^{-2|u_j|}\right)
		du_1\cdots du_k 
		\\ =& \left(S(I)^2/{\sqrt{4-S(I)^2}} \right)^k\prod_{j=1}^k\left(e^{c_0-x_j}\right)
		=(M(I))^k\prod_{j=1}^k\left(e^{c_0-x_j}\right).\end{align*} For the cases when $a+b>0$ and $a+b=0$, the proof follows similarly. This completes the proof of the lower bound \eqref{34}, and hence Theorem \ref{thm2}.

	All of the rest effort is to prove Lemma \ref{lem24}. We first need a   lower bound of $\mathbb{P}(\xi^{(n)}(J)=0) $ when $J$ is a finite union of intervals.\begin{lem}\label{lem25}Let $ \varepsilon_0\in(0,1),\ C_0>0,\ k\in\mathbb{Z^+},\ I=[a,b]=[y_0,y_{k+1}]\subset(-2,2).$ Assume $y_1,\cdots,y_k\in I,\ a_1,\cdots,a_k\in (G_n(-C_0)/S(I),G_n(C_0)/S(I))\cap(0,\varepsilon_0(2\ln n)^{-1}),$ $|y_i-y_j|\geq\varepsilon_0(\ln n)^{-1}$ for every $0\leq i<j\leq k+1,$ $\sqrt{4-y_i^2}/S(I)\leq 1+C_0(\ln n)^{-1}$ for every $1\leq i\leq k.$ Then there exists a constant $N_5>0$ depending only on $\varepsilon_0,C_0,k,I$ such that for $n>N_5$ we have\begin{align*}&\mathbb{P}(\xi^{(n)}(\cup_{j=1}^k[y_j,y_j+a_j])=0)\geq(1-(\ln n)^{-1})\prod_{j=1}^kD_n\left(a_j\sqrt{4-y_j^2}/2\right).
	\end{align*}\end{lem}\begin{proof} We use $f=O(g)$ to denote $|f|\leq Cg$ for a constant $C$ depending only on $\varepsilon_0,C_0,k,I.$ As $|y_i-y_j|\geq\varepsilon_0(\ln n)^{-1}>\varepsilon_0(2\ln n)^{-1}$ for $i\neq j$, if $1\leq j\leq k,$ then $ y_0\leq y_j<y_j+a_j<y_j+\varepsilon_0(2\ln n)^{-1}<y_j+|y_{k+1}-y_j|= y_{k+1}$, and thus $[y_j,y_j+a_j]\subset[y_0,y_{k+1}]=I. $ If $1\leq i<j\leq k,$ then by assumption $a_i,a_j\in(0,\varepsilon_0(2\ln n)^{-1})\subset(0,|y_i-y_j|), $ and thus $[y_i,y_i+a_i]\cap[y_j,y_j+a_j]=\emptyset. $
		Therefore, we have $J:=\cup_{j=1}^k[y_j,y_j+a_j]$ is a disjoint union and $J\subset I. $ Let's denote $$A=\chi_J P_{GUE(n)}\chi_J,\,\,\,\,A_{i,j}=\chi_{ [y_i,y_i+a_i]} P_{GUE(n)}\chi_{[y_j,y_j+a_j]},$$ then we have $$A=\sum\limits_{i=1}^k\sum\limits_{j=1}^k A_{i,j}$$  and
		\begin{align}\label{66}&\mathbb{P}(\xi^{(n)}(\cup_{j=1}^k[y_j,y_j+a_j])=0)=\mathbb{P}(\xi^{(n)}(J)=0)=\det(\text{Id}-A).
		\end{align}Let  $B_j$ be the integral operator with
		kernel\begin{align*}&B_j (u,v)={2\pi}\rho_{sc}(y_j)K_{(y_j,y_j+a_j)}^{CUE(n)}\left({2\pi}\rho_{sc}(y_j)u,{2\pi}\rho_{sc}(y_j)v\right),
		\end{align*} where $K^{CUE(n)}(x,y)$ is the kernel defined in \eqref{kern}.
		Let's denote \begin{align*}&B=\sum\limits_{j=1}^kB_j.
		\end{align*}As $ 0<a_j\sqrt{4-y_j^2}/2\leq a_j<\varepsilon_0(2\ln n)^{-1}<1$, we have \begin{align}\label{68}\det(\text{Id}-B_j)&=\mathbb{P}^{CUE(n)}(\theta_i\not\in [0,2\pi\rho_{sc}(y_j)a_j],1\leq i\leq n)\\ \nonumber&=D_n(\pi\rho_{sc}(y_j)a_j)=D_n\left(a_j\sqrt{4-y_j^2}/2\right),
		\end{align}and \begin{align}\label{69}\det(\text{Id}-B)=\prod_{j=1}^k\det(\text{Id}-B_j)=\prod_{j=1}^kD_n\left(a_j\sqrt{4-y_j^2}/2\right).
		\end{align}Now we need to compare the Fredholm determinants, the key point is to estimate $|A-B|_2,\ \Tr(A-B),\ \|(\text{Id}-B)^{-1}\|.$ Comparing the support of the kernels, we have\begin{align}\label{70}&|A-B|_2^2=\sum\limits_{j=1}^k|A_{j,j}-B_j|_2^2+\sum\limits_{i\neq j}|A_{i,j}|_2^2,\ \Tr(A-B)=\sum\limits_{j=1}^k\Tr(A_{j,j}-B_j).
		\end{align}For $x\in [y_i, y_i+a_i]\subset I,\ y\in [y_j, y_j+a_j]\subset I,\ i\neq j,\ (1\leq i,j\leq k)$, we have\begin{align*}&|x-y|\geq |y_i-y_j|-\max(a_i,a_j)\geq \varepsilon_0(\ln n)^{-1}-\varepsilon_0(2\ln n)^{-1}=\varepsilon_0(2\ln n)^{-1}.
		\end{align*}From the Plancherel-Rotach asymptotics for the Hermite polynomials
		(Theorem 8.22.9 in \cite{Sz}) for any nonnegative integer $j,\ \psi_{n-j}(\sqrt{n}x)$ is $O(n^{-1/4}),$ uniformly in $x\in I.$ Consequently, if $|x-y| \geq \varepsilon_0(2\ln n)^{-1},\ x,y\in I,$ from \eqref{29}, we have  \begin{align*}&|K^{GUE(n)}(x,y)|=\sqrt{n}\frac{O(n^{-1/4})O(n^{-1/4})}{|x-y|}=\frac{O(1)}{|x-y|}\leq \frac{O(1)}{\varepsilon_0(2\ln n)^{-1}}=O(\ln n).\end{align*}Using this and \eqref{31}, for $i\neq j$ we have (recall that $0<a_i<G_n(C_0)/S(I)$)\begin{align}\label{71}|A_{i,j}|_2^2=&\int_{[y_i, y_i+a_i]}dx\int_{[y_j, y_j+a_j]}|K^{GUE(n)}(x,y)|^2dy\\ \nonumber =&\int_{[y_i, y_i+a_i]}dx\int_{[y_j, y_j+a_j]}O((\ln n)^2)dy=a_ia_jO((\ln n)^2)\\ \nonumber\leq&(G_n(C_0)/S(I))^2O((\ln n)^2)=O\left(\frac{\ln n}{n^2}\right)O((\ln n)^2)=O\left(\frac{(\ln n)^3}{n^2}\right).
		\end{align}Since $a_j=O(\varepsilon_0(2\ln n)^{-1})=o(1),\ 0<S(I)\leq \sqrt{4-y_j^2}={2\pi}\rho_{sc}(y_j)\leq 2,$ and the kernel of $A_{j,j}$ is $A_{j,j} (u,v)=K_{(y_j,y_j+a_j)}^{GUE(n)}(u,v),
		$ by Lemma 3.4 in \cite{BB} we have\begin{align*}&\frac{1}{n\rho_{sc}(x)}K^{GUE(n)}(x,y)-\frac{\sin(n\pi \rho_{sc}(x)(x-y) )}{n\pi \rho_{sc}(x)(x-y)}=O\left(\frac{1}{n}\right)+O\left(a_j\right)+O\left(na_j^2\right),\\
			&\frac{2\pi}{n}K^{CUE(n)}({2\pi}\rho_{sc}(y_j)x,{2\pi}\rho_{sc}(y_j)y)-\frac{\sin(n\pi \rho_{sc}(y_j)(x-y) )}{n\pi \rho_{sc}(y_j)(x-y)}=O\left(\frac{a_j}{n}\right),\end{align*}uniformly for $x,y\in[y_j,y_j+a_j].$ Thus the difference between the two kernels $A_{j,j}$ and $B_j$ is $O(1+n^2a_j^2),$ integrating on a domain $[y_j,y_j+a_j]^2 $ of area $a_j^2,$ we have \begin{align*}&|A_{j,j}-B_j|_2^2=O((1+n^2a_j^2)^2)a_j^2=O(a_j^2+n^4a_j^6);\end{align*} and integrating on the diagonal $\{x=y\in[y_j,y_j+a_j]\} $ yields\begin{align*}&|\Tr(A_{j,j}-B_j)|=O((1+n^2a_j^2))a_j=O((a_j^2+n^4a_j^6)^{1/2}).\end{align*}Using $0<a_j<G_n(C_0)/S(I)$ and  \eqref{31}, we have \begin{align}\label{76}&a_j^2\leq (G_n(C_0)/S(I))^2=O\left(\frac{\ln n}{n^2}\right),\end{align}thus\begin{align}\label{72}&|A_{j,j}-B_j|_2^2=O(a_j^2+n^4a_j^6)=O\left(\frac{\ln n}{n^2}+\frac{(\ln n)^3}{n^2}\right)=O\left(\frac{(\ln n)^3}{n^2}\right),\end{align}and\begin{align}\label{73}&|\Tr(A_{j,j}-B_j)|=O((a_j^2+n^4a_j^6)^{1/2})=O\left(\frac{(\ln n)^{3/2}}{n}\right).\end{align}Using \eqref{70}\eqref{71}\eqref{72}\eqref{73}, we conclude that\begin{align}\label{74}&|A-B|_2^2=O\left(\frac{(\ln n)^3}{n^2}\right),\ |\Tr(A-B)|=O\left(\frac{(\ln n)^{3/2}}{n}\right).
		\end{align}Recall the formula  \eqref{kern}, we have $K^{CUE(n)}(x,x)=\dfrac{n}{2\pi}$ and $$|K^{CUE(n)}(x,y)|=O\left(\dfrac{n}{1+n|x-y|}\right),\,\,\, |x-y|\leq 2.$$ Therefore, by definition of $B_j$, we have $$B_j(u,u)={2\pi}\rho_{sc}(y_j)\dfrac{n}{2\pi}={n\rho_{sc}(y_j)}, \,\, u\in(y_j,y_j+a_j)$$ and\begin{align}\label{75}&\Tr B_j=\int_{y_j}^{y_j+a_j} B_j(u,u)du={na_j\rho_{sc}(y_j)}=na_i\sqrt{4-y_i^2}/(2\pi);
		\end{align}since $0<2\pi \rho_{sc}(y_j)a_j=\sqrt{4-y_j^2}a_j\leq2a_j<2$ and $0<S(I)\leq \sqrt{4-y_j^2}={2\pi}\rho_{sc}(y_j)\leq 2$, thus we have the off-diagonal estimate $$|B_j(u,v)|=O\left(\dfrac{n}{1+n|u-v|}\right),\,\, u,v\in(y_j,y_j+a_j).$$Therefore, we have  \begin{align*}&|B_j|_2^2=\int_{y_j}^{y_j+a_j}\int_{y_j}^{y_j+a_j} |B_j(u,v)|^2dudv\\=&\int_{y_j}^{y_j+a_j}\int_{y_j}^{y_j+a_j} O\left(\dfrac{n^2}{(1+n|u-v|)^2}\right)dudv\\=&\int_{y_j}^{y_j+a_j}O\left(\int_{\mathbb{R}} \dfrac{n^2}{(1+n|u-v|)^2}du\right)dv\\=&\int_{y_j}^{y_j+a_j}O\left(n\right)dv=O(na_j)=O\left((\ln n)^{1/2}\right),
		\end{align*}here we used \eqref{76}. Therefore,  we have \begin{align}\label{77}&|B|_2^2=\sum\limits_{j=1}^k|B_j|_2^2=O\left((\ln n)^{1/2}\right).
		\end{align} Now we estimate $\|(\text{Id}-B)^{-1}\|.$ We have $\|(\text{Id}-B)^{-1}\|=(1-\lambda_1(B))^{-1} $ where $\lambda_1(B) $ is the largest eigenvalue of $B$. Similar to the CUE case as in \eqref{simd}, we know that $\lambda_1(B)\leq\lambda_1(B_i)$ for some $1\leq i\leq k$ or $\lambda_1(B)=0.$ For every $1\leq i\leq k$, by Lemma \ref{lem18}, \eqref{68} and \eqref{75},  we have $$1-\lambda_1(B_i) \geq \det(\text{Id} -B_i)e^{\Tr B_i-1}=D_n\left(a_i\sqrt{4-y_i^2}/2\right)e^{na_i\sqrt{4-y_i^2}/(2\pi)-1}.$$
		By \eqref{31}\eqref{21} and $32>\pi^2$,  there exists a constant $N_{5,0}>0$ such that $$\pi(\ln n)^{\frac{1}{2}}< nG_n(-C_0),$$
		and
		$$n(\ln n)^{-\frac{1}{2}} D_n((1+C_0/\ln n)G_n(C_0)/2)>e^{c_0-3C_0}$$ and  $G_n(C_0)<1,\ C_0<\ln n $ for $n>N_{5,0}.$ By assumption, $a_i<G_n(C_0)/S(I)$ and\\ $\sqrt{4-y_i^2}/S(I)\leq 1+C_0(\ln n)^{-1}$, we have $a_i\sqrt{4-y_i^2}<(1+C_0/\ln n)G_n(C_0). $ Since\\ $a_i>G_n(-C_0)/S(I),$ $\sqrt{4-y_i^2}/S(I)\geq 1$  for $y_i\in I$, we have $a_i\sqrt{4-y_i^2}>G_n(-C_0). $ Thus if $n>N_{5,0},\ 1\leq i\leq k$, we have \begin{align*}1-\lambda_1(B_i) &\geq D_n\left(a_i\sqrt{4-y_i^2}/2\right)e^{na_i\sqrt{4-y_i^2}/(2\pi)-1}\\ &\geq D_n\left((1+C_0/\ln n)G_n(C_0)/2\right)e^{nG_n(-C_0)/(2\pi)-1}\\&\geq n^{-1}(\ln n)^{\frac{1}{2}}e^{c_0-3C_0}e^{(\ln n)^{\frac{1}{2}}/2-1}.\end{align*}Now we always assume $n>N_{5,0},$ then similar to the CUE case, we have 
\begin{align}\label{78}&\|(\text{Id}-B)^{-1}\|=(1-\lambda_1(B))^{-1}\leq\max_{1\leq i\leq k}(1-\lambda_1(B_i))^{-1}+1 \\ \nonumber\leq& n(\ln n)^{-\frac{1}{2}}e^{3C_0-c_0-(\ln n)^{\frac{1}{2}}/2+1}+1=O\left(n(\ln n)^{-\frac{1}{2}}e^{-(\ln n)^{\frac{1}{2}}/2}\right).\end{align}By Lemma \ref{lem15} and \eqref{74}\eqref{77}\eqref{78}, we conclude that\begin{align*}b_2:=&|\Tr((A-B)(\text{Id}-B)^{-1})|\leq|\Tr(A-B)|+|A-B|_2|B|_2\|(\text{Id}-B)^{-1}\|\\ \leq& O\left(\frac{(\ln n)^{3/2}}{n}\right)+O\left(\frac{(\ln n)^{3/2+1/4}}{n}\right)O\left(n(\ln n)^{-\frac{1}{2}}e^{-(\ln n)^{\frac{1}{2}}/2}\right)\\ =& O\left(\frac{(\ln n)^{3/2}}{n}\right)+O\left({(\ln n)^{5/4}}e^{-(\ln n)^{\frac{1}{2}}/2}\right)=O\left({(\ln n)^{-2}}\right),\end{align*}that\begin{align*}b_3:=&|B-A|_2^2\|(\text{Id}-B)^{-1}\|^2 \leq O\left(\frac{(\ln n)^{3}}{n^2}\right)O\left(n^2(\ln n)^{-1}e^{-(\ln n)^{\frac{1}{2}}}\right)\\ =& O\left({(\ln n)^{2}}e^{-(\ln n)^{\frac{1}{2}}}\right)=O\left({(\ln n)^{-2}}\right).\end{align*}Therefore, by Lemma \ref{lem15} again, we have \begin{align*}1-b_3=&1-|B-A|_2^2\|(\text{Id}-B)^{-1}\|^2\\ \leq&\exp(\Tr(A-B)(\text{Id}-B)^{-1})\det(\text{Id} -A)/\det(\text{Id}-B)\\ \leq& e^{b_2}\det(\text{Id} -A)/\det(\text{Id}-B).\end{align*}Thus there exists a constant $N_{5}>N_{5,0}$ such that $b_2<(2\ln n)^{-1}<1,\ b_3<(2\ln n)^{-1}<1$ for $N>N_5$ and\begin{align}\label{79} \det(\text{Id} -A)/\det(\text{Id}-B)&\geq e^{-b_2}(1-b_3)\geq (1-b_2)(1-b_3)\\ \nonumber &\geq (1-(2\ln n)^{-1})^2\geq 1-(\ln n)^{-1},\ \forall\ n>N_5.\end{align}Now the result follows from \eqref{66}\eqref{69} and \eqref{79}.\end{proof}
	
	Now we prove Lemma \ref{lem24}.\begin{proof}Let $u_0=0$ and \begin{align}\label{80}C_0=\max\limits_{1\leq j\leq k}(|x_j|+|u_j|),\ \varepsilon_1=\min\limits_{0\leq i<j\leq k}\big||u_i|-|u_j|\big|,\ \varepsilon_0=\varepsilon_1S(I)^2/(2+4\varepsilon_1).\end{align} Using $l_n=(b_*/a_*-1)\ln n\to+\infty$ and \eqref{31}, there exists a constant $N_{6,0}>2$ such that $l_n>C_0$ and $0<4(\ln n)^{\frac{1}{2}}/n< G_n(-C_0)\leq G_n(C_0)<8(\ln n)^{\frac{1}{2}}/n $ for $n>N_{6,0}$. Let's denote \begin{align}\label{81}y_j=\gamma_n(u_j),\ a_j=G_n(x_j)/S(I),\ \ \forall\ n>N_{6,0}.\end{align} Then we have $y_j\in (a,b)$ for $1\leq j\leq k,\ n>N_{6,0}$ (See the range of $\gamma_n $ in {\bf Case 1}-{\bf Case 3}). Now we need to check all  assumptions in Lemma \ref{lem25}.
		
		(a) Since $u_j\neq 0$, we have $C_0>0$. Since $|u_1|,\cdots,|u_k|$ are nonzero and all distinct in all the 3 cases, we have $ \varepsilon_1>0.$ Using this and $0<S(I)\leq 2, $ we have $$0<\varepsilon_0=\varepsilon_1S(I)^2/(2+4\varepsilon_1)\leq 4\varepsilon_1/(2+4\varepsilon_1)<1 .$$
		
		(b) By \eqref{57}\eqref{63}, we have $(\gamma_n(u))^2=4-(1+|u|/\ln n)^2S(I)^2 $. Thus by \eqref{81}, we have $y_j^2=(\gamma_n(u_j))^2=4-(1+|u_j|/\ln n)^2S(I)^2$ and \begin{align}\label{82}\sqrt{4-y_j^2}=(1+|u_j|/\ln n)S(I),\ \sqrt{4-y_j^2}a_j=(1+|u_j|/\ln n)G_n(x_j).\end{align} For $1\leq i<j\leq k,\ n>N_{6,0}, $ we have $y_i,y_j\in(a,b)\subset(-2,2),\ |y_i+y_j|<4$ and\begin{align*}&4|y_i-y_j|\geq |y_i+y_j|\cdot|y_i-y_j|=\left|y_i^2-y_j^2\right|=\left|(\gamma_n(u_i))^2-(\gamma_n(u_j))^2\right|\\=&
			\left|(1+|u_j|/\ln n)^2-(1+|u_i|/\ln n)^2\right|S(I)^2\\=&
			\big||u_j|-|u_i|\big|/\ln n \cdot (2+|u_j|/\ln n+|u_i|/\ln n)S(I)^2\\ \geq &
			\big||u_j|-|u_i|\big|/\ln n\cdot2S(I)^2\geq \varepsilon_1/\ln n\cdot2S(I)^2,\end{align*}thus we have $$|y_i-y_j|\geq \varepsilon_1S(I)^2/(2\ln n)=\varepsilon_0 (1+2\varepsilon_1)(\ln n)^{-1}\geq \varepsilon_0 (\ln n)^{-1}.$$
		Actually, the similar arguments apply to the end points $y_0$ and $y_{k+1}$ and we finally have $|y_i-y_j|\geq \varepsilon_0 (\ln n)^{-1}$ for
		$0\leq i<j\leq k+1$.
		
		(c) For every $1\leq i\leq k$ and $n>N_{6,0}, $ we have $|u_i|\leq C_0$, and by \eqref{82}, we have\begin{align*}&\sqrt{4-y_i^2}/S(I)=(1+|u_i|/\ln n)S(I)/S(I)=1+|u_i|/\ln n\leq 1+C_0/\ln n.\end{align*}
		
		
		(d) For every $1\leq j\leq k$ and $n>N_{6,0}, $ since $a_j=G_n(x_j)/S(I),\ S(I)>0,\ |x_j|<|x_j|+|u_j|\leq C_0$ and $G_n$ is increasing, we have $0<4(\ln n)^{\frac{1}{2}}/n< G_n(-C_0)< G_n(x_j)=a_jS(I)< G_n(C_0)<8(\ln n)^{\frac{1}{2}}/n $ and thus $a_j\in(G_n(-C_0)/S(I),G_n(C_0)/\\S(I)) \cap(0,8(\ln n)^{\frac{1}{2}}/(nS(I))).$ Since $\varepsilon_0>0,\ S(I)>0,$ there exists a constant $N_{6,1}>N_{6,0}$ such that $16(\ln n)^{\frac{3}{2}}/n<\varepsilon_0S(I) $ for $n>N_{6,1}.$ Thus $8(\ln n)^{\frac{1}{2}}/(nS(I))<\varepsilon_0(2\ln n)^{-1} $ and we have $a_j\in(0,\varepsilon_0(2\ln n)^{-1})$ for $1\leq j\leq k$ and $n>N_{6,1}. $
		
		From the statements (a)-(d), we know that $\varepsilon_0,\ C_0 $ defined in \eqref{80} and \eqref{81} satisfy all the assumptions in Lemma \ref{lem25} for $n>N_{6,1}$. Thus $[y_i,y_i+G_n(x_i)/S(I)]\cap [y_j,y_j+G_n(x_j)/S(I)]=[y_i,y_i+a_i]\cap[y_j,y_j+a_j]=\emptyset$ for every $1\leq j\leq k$ and $n>N_{6,1},$ then $(y_1,\cdots,y_n)\in A_n$ (recall \eqref{andd}) for $n>N_{6,1}$ and we can use the notation \eqref{58} and formula \eqref{59} in this case. For $n>N_{6,1}$, by \eqref{59} we have \begin{align}\label{90}&(\ln n)^{-k}\phi_{k,n}(y_1,\cdots,y_k)\geq (2n)^k(2\ln n)^{-\frac{k}{2}}\mathbb{P}(\xi^{(n)}(I_{n,k})=0)\\ \nonumber&-(2n)^k(2\ln n)^{-\frac{k}{2}}\sum_{j=0}^k\mathbb{P}(\xi^{(n)}(I_{n,k})=\xi^{(n)}(J_{n,k,j})=0).
		\end{align} As in the CUE case, we claim that $$(2n)^k(2\ln n)^{-\frac{k}{2}}\mathbb{P}(\xi^{(n)}(I_{n,k})=\xi^{(n)}(J_{n,k,j})=0) \to 0.$$  Since $a_j=G_n(x_j)/S(I)$, by \eqref{58} we have $I_{n,k}=\cup_{j=1}^k[y_j,y_j+a_j]$. Let\\ $d_0:=G_n(-C_0)/S(I),$ then we have $0<d_0<a_j<\varepsilon_0(2\ln n)^{-1}$ for $1\leq j\leq k$ and $n>N_{6,1}. $ Let $z_j'=(z_j+z_{j+1})/2$ for $0\leq j\leq k$ where $z_j\ (0\leq j\leq k+1)$ is the increasing rearrangement of $y_j\ (0\leq j\leq k+1)$. Since $y_j\in I\ (0\leq j\leq k+1)$, we have $z_j\in I\ (0\leq j\leq k+1),\ z_j'\in I\ (0\leq j\leq k)$ and \begin{align*}&\min_{0\leq i\leq k+1}|z_j'-z_i|=z_{j+1}-z_j'=z_j'-z_j=(z_{j+1}-z_j)/2\\ &\geq\min_{0\leq i<l\leq k+1}|y_i-y_l|/2\geq \varepsilon_0(2\ln n)^{-1}>d_0\end{align*}for $0\leq j\leq k$ and $n>N_{6,1}. $ Thus $ [z_j',z_j'+d_0]\cap[z_i,z_i+d_0]=\emptyset$ and  $ [z_j',z_j'+d_0]\subset [z_j,z_{j+1}]=J_{n,k,j}$ for $0\leq j\leq k,\ 0\leq i\leq k$ and $n>N_{6,1}. $ Since $d_0<a_j$,  we have $I_{n,k}\supseteq\cup_{j=1}^k[y_j,y_j+d_0]=\cup_{j=1}^k[z_j,z_j+d_0] $, and $I_{n,k}\cup J_{n,k,j}\supseteq [z_j',z_j'+d_0]\cup \left(\cup_{i=1}^k[z_i,z_i+d_0]\right), $ and the right hand side is a disjoint union for $n>N_{6,1}$. By \eqref{16} we have \begin{align*}0\leq&\mathbb{P}(\xi^{(n)}(I_{n,k})=\xi^{(n)}(J_{n,k,j})=0)=\mathbb{P}(\xi^{(n)}(I_{n,k}\cup J_{n,k,j})=0)\\ \leq&\mathbb{P}([z_j',z_j'+d_0]\cup \left(\cup_{i=1}^k[z_i,z_i+d_0]\right))=0)\\ \leq&\mathbb{P}(\xi^{(n)}([z_j',z_j'+d_0])=0)\prod_{i=1}^k\mathbb{P}(\xi^{(n)}([z_i,z_i+d_0])=0)\leq p_{n,k}^{k+1},\end{align*}where \begin{align*}p_{n,k}:=\sup_{z\in I}\mathbb{P}(\xi^{(n)}([z,z+d_0])=0)=\sup_{z\in I}\mathbb{P}(\xi^{(n)}([z,z+G_n(-C_0)/S(I)])=0).
		\end{align*}By  \eqref{21} and \eqref{32}, there exists a constant $N_{6,2}>N_{6,1}$ such that\begin{align*}&n(2\ln n)^{\frac{1}{2}}\sup_{z\in I}\Big(\mathbb{P}(\xi^{(n)}([z,z+G_n(-C_0)/S(I)])=0)\\ &-D_n(\sqrt{4-z^2}/S(I) \cdot G_n(-C_0)/2)\Big)<1,
		\end{align*}and \begin{align*}&n(2\ln n)^{-\frac{1}{2}}D_n(  G_n(-C_0)/2)<e^{c_0+C_0+1}.
		\end{align*} Then we further have\begin{align*}&p_{n,k}=\sup_{z\in I}\mathbb{P}(\xi^{(n)}([z,z+G_n(-C_0)/S(I)])=0)\\ &\leq\sup_{z\in I}\Big(\mathbb{P}(\xi^{(n)}([z,z+G_n(-C_0)/S(I)])=0)-D_n(\sqrt{4-z^2}/S(I) \cdot G_n(-C_0)/2)\Big)\\&+\sup_{z\in I}D_n(\sqrt{4-z^2}/S(I) \cdot G_n(-C_0)/2)\\ &\leq n^{-1}(2\ln n)^{-\frac{1}{2}}+D_n( G_n(-C_0)/2)\leq n^{-1}(2\ln n)^{\frac{1}{2}}+n^{-1}(2\ln n)^{\frac{1}{2}}e^{c_0+C_0+1},
		\end{align*} where we used the fact that $D_n(\alpha)$ is decreasing.  Thus, for every $0\leq j\leq k$, we have,  \begin{align*}&\limsup_{n\to+\infty}(2n)^k(2\ln n)^{-\frac{k}{2}}\mathbb{P}(\xi^{(n)}(I_{n,k})=\xi^{(n)}(J_{n,k,j})=0)\\ \leq &\limsup_{n\to+\infty}(2n)^k(2\ln n)^{-\frac{k}{2}}p_{n,k}^{k+1}\\ \leq&\limsup_{n\to+\infty}(2n)^k(2\ln n)^{-\frac{k}{2}}\left(n^{-1}(2\ln n)^{\frac{1}{2}}+n^{-1}(2\ln n)^{\frac{1}{2}}e^{c_0+C_0+1}\right)^{k+1}\\ \leq&\limsup_{n\to+\infty}2^kn^{-1}(2\ln n)^{\frac{1}{2}}\left(1+e^{c_0+C_0+1}\right)^{k+1}=0,
		\end{align*} which completes the claim. 
		
		Now using \eqref{21}\eqref{82}\eqref{90} and Lemma \ref{lem25}, we have \begin{align*}&\liminf_{n\to+\infty}(\ln n)^{-k}\phi_{k,n}(y_1,\cdots,y_k)\\\geq& \liminf_{n\to+\infty}(2n)^k(2\ln n)^{-\frac{k}{2}}\mathbb{P}(\xi^{(n)}(I_{n,k})=0)\\=&\liminf_{n\to+\infty}(2n)^k(2\ln n)^{-\frac{k}{2}}\mathbb{P}(\xi^{(n)}(\cup_{j=1}^k[y_j,y_j+a_j])=0) \\ \geq &\liminf_{n\to+\infty}(2n)^k(2\ln n)^{-\frac{k}{2}}(1-(\ln n)^{-1})\prod_{j=1}^kD_n\left(a_j\sqrt{4-y_j^2}/2\right)\\=&\liminf_{n\to+\infty}(2n)^k(2\ln n)^{-\frac{k}{2}}\prod_{j=1}^k D_n((1+|u_j|/\ln n)G_n(x_j)/2))\\=&2^k\prod_{j=1}^k\left(\lim_{n\to+\infty}n(2\ln n)^{-\frac{1}{2}}D_n((1+|u_j|/\ln n)G_n(x_j)/2)\right)\\=&2^k\prod_{j=1}^k\left(e^{c_0-x_j-2|u_j|}\right)=2^ke^{\sum_{j=1}^k(c_0-x_j-2|u_j|)}.
		\end{align*}Now  Lemma \ref{lem24} follows from the definition of $y_j$ in \eqref{81}.\end{proof}

 \end{document}